\documentclass[11pt]{article}
\usepackage{graphicx}
 \usepackage{mathptmx}      
\usepackage[titletoc,title]{appendix}
\usepackage[numbers]{natbib}
\usepackage[T1]{fontenc}
\usepackage[applemac]{inputenc}
\usepackage[all]{xy}
\usepackage{enumerate}
\usepackage{comment}
\usepackage{amsmath,amssymb,color}
\usepackage[british]{babel}
\usepackage{a4wide}
\usepackage[sans]{dsfont}
\usepackage{amsfonts}
\usepackage{amsthm}
\usepackage[mathscr]{euscript}
\usepackage{graphicx}
\usepackage{mathrsfs}

\usepackage{latexsym}

\usepackage{amsthm,amsmath,natbib}

\newcommand{\p}{^}

	\newcommand{\rmd}{\mathrm{d}}
\newcommand{\rme}{\mathrm{e}}
\newcommand{\rmi}{\mathrm{i}}

\newcommand{\Rea}{\mathrm{Re}}
\newcommand{\Ima}{\mathrm{Im}}

\newcommand{\E}{\mathrm{E}}
\newcommand{\M}{\mathrm{M}}

\newcommand{\Lrm}{\mathrm{L}}

\newcommand{\Var}{\mathrm{Var}}

\newcommand{\Ascr}{\mathscr{A}}
\newcommand{\Bscr}{\mathscr{B}}
\newcommand{\Cscr}{\mathscr{C}}

\newcommand{\Escr}{\mathscr{E}}
\newcommand{\Fscr}{\mathscr{F}}
\newcommand{\Gscr}{\mathscr{G}}
\newcommand{\Hscr}{\mathscr{H}}

\newcommand{\Lscr}{\mathscr{L}}

\newcommand{\Oscr}{\mathscr{O}}
\newcommand{\Pscr}{\mathscr{P}}

\newcommand{\Sscr}{\mathscr{S}}

\newcommand{\ka}{\kappa}
\newcommand{\zt}{\zeta}
\newcommand{\et}{\eta}

\newcommand{\al}{\alpha}
\newcommand{\bt}{\beta}
\newcommand{\ep}{\varepsilon}
\newcommand{\lm}{\lambda}
\newcommand{\gm}{\gamma}

\newcommand{\sig}{\sigma}

\newcommand{\om}{\omega}
\newcommand{\Om}{\Omega}
\newcommand{\cadlag}{c\`adl\`ag }

\newcommand{\Ebb}{\mathbb{E}}
\newcommand{\EE}{\mathbb{E}}
\newcommand{\Fbb}{\mathbb{F}}

\newcommand{\Pbb}{\mathbb{P}}

\newcommand{\Rbb}{\mathbb{R}}

\newcommand{\Cbb}{\mathbb{C}}

\newcommand{\bi}{\begin{itemize}}
\newcommand{\ei}{\end{itemize}}

\newcommand{\be}{\begin{enumerate}}
\newcommand{\ee}{\end{enumerate}}
\newcommand{\beq}{\begin{equation}}
\newcommand{\eeq}{\end{equation}}
\newcommand{\beqs}{\begin{equation*}}
\newcommand{\eeqs}{\end{equation*}}
\newcommand{\beqa}{\begin{eqnarray}}
\newcommand{\eeqa}{\end{eqnarray}}
\newcommand{\beqas}{\begin{eqnarray*}}
\newcommand{\eeqas}{\end{eqnarray*}}

\newcommand{\aPP}[2]{\ensuremath{\langle #1,#2 \rangle}}

\newtheorem{theorem}{Theorem}[section]
\newtheorem{lemma}[theorem]{Lemma}
\newtheorem{proposition}[theorem]{Proposition}
\newtheorem{corollary}[theorem]{Corollary}
\theoremstyle{definition}
\newtheorem{definition}[theorem]{Definition}

\newtheorem{assumption}[theorem]{Assumption}

\newtheorem{remark}[theorem]{Remark}

\numberwithin{equation}{section}

\newcommand{\RR}{\mathbb{R}}

\begin{document}
\title{Semi-Static Variance-Optimal Hedging\\ in Stochastic Volatility Models\\ with Fourier Representation}
\author{P.\ Di Tella\footnote{Corresponding author. Mail: Paolo.Di\_Tella@tu-dresden.de.}, M.\ Haubold, M.\ Keller-Ressel\thanks{MKR thanks Johannes Muhle-Karbe for early discussions on the idea of ``Variance-Optimal Semi-Static Hedging''. We acknowledge funding from the German Research Foundation (DFG) under grant ZUK 64 (all authors) and KE 1736/1-1 (MKR, MH)}\\\small{Technische Universit\"at Dresden, Intitut f\"ur Stochastik}}

\date{}
\maketitle
\begin{abstract}
In a financial market model, we consider the variance-optimal semi-static hedging of a given contingent claim, a generalization of the classic variance-optimal hedging. To obtain a tractable formula for the expected squared hedging error and the optimal hedging strategy, we use a Fourier approach in a general multidimensional semimartingale factor model. As a special case, we recover existing results for variance-optimal hedging in affine stochastic volatility models. We apply the theory to set up a variance-optimal semi-static hedging strategy for a variance swap in both the Heston and the $3/2$-model, the latter of which is a non-affine stochastic volatility model. 
\end{abstract}
\section{Introduction}\label{sec:intro}
Variance-optimal hedging was introduced by \cite{S84, FS86} as a tractable method of hedging contingent claims in incomplete markets (cf.~\cite{Sch99} for an overview). The main idea of variance-optimal hedging is to find a self-financing trading strategy $\vartheta$ for a given claim $\eta^0$, which minimizes the (risk-neutral) variance of the residual hedging error at a terminal time $T > 0$. As shown in \cite{FS86}, the solution of this optimization problem is given by the so-called Galtchouk--Kunita--Watanabe (GKW) decomposition of $\eta^0$ with respect to the underlying stock $S$.  The GKW decomposition takes the form 
\[
\eta^0= \EE[\eta^0] + (\vartheta \cdot S)_T + L^0_T\,,
\]
where $L_T$ is the terminal value of a local martingale that is orthogonal to the underlying $S$. For applications it is important to have tractable formulas for the variance-optimal strategy $\vartheta$ and for the minimal expected hedging error $\epsilon^2 := \EE[(L_T^0)^2]$. To this end, several authors have combined variance-optimal hedging with \emph{Fourier methods}. This idea was first pursued in \cite{HK06}, where in an exponential L\'evy market the expected hedging error has been explicitly computed for contingent claims with integrable and invertible Laplace transform. This method has been further investigated in \cite{KalP07}, where the underlying market model is a stochastic volatility model described by an \emph{affine semimartingale}.

In this paper we extend the results of \cite{HK06, KalP07} in two directions: First, we consider a very general setting of semimartingale factor models that is not limited to processes with independent increments or with affine structure. Second, in addition to classic variance optimal hedging, we also consider the variance-optimal \emph{semi-static} hedging problem that we have introduced in \cite{DTHKR17}. The semi-static hedging problem combines dynamic trading in the underlying $S$ with static (i.\ e.\ buy-and-hold) positions in a finite number of given contingent claims $(\eta^1, \dotsc, \eta^d)$, e.\ g.\  European puts and calls. Such semi-static hedging strategies have been considered for the hedging of Barrier options (cf.~\cite{carr2011semi}), in model-free hedging approaches based on martingale optimal transport (cf.~\cite{beiglbock2013model, B17}), and finally for the semi-static replication of variance swaps by Neuberger's formula (cf.~\cite{N94}).

As shown in \cite{DTHKR17} and summarized in Section~\ref{sec:GKW.dec} below, the semi-static hedging problem can be solved under a variance-optimality criterion when also the \emph{covariances} $\epsilon_{ij} := \EE[L_T^i L_T^j]$ of the residuals in the GKW-decompositions of all supplementary claims $(\eta^1, \dotsc, \eta^d)$ are known. This leads to a natural extension of the questions investigated in \cite{HK06, KalP07}: If the model for $S$ allows for explicit calculation of the characteristic function $S = \log X$ (such as L\'evy models, the Heston model or the 3/2 model) and the supplementary claims $(\eta^1, \dotsc, \eta^d)$ have a Fourier representation (such as European puts and calls), how can we compute the quantities $\epsilon_{ij}$? 

In our main results, Theorems~\ref{thm:hed.err}, \ref{thm:pred.cov.mixt} and \ref{thm:comsec3.sec4} we provide answers to this question and show that the covariances $\epsilon_{ij} := \E[L_T^i L_T^j]$ and the optimal strategies can be recovered by Fourier-type integrals, extending the results of \cite{HK06, KalP07}. In addition, these results serve as the rigorous mathematical underpinning of the more applied point of view taken in \cite{DTHKR17}.

This paper has the following structure: In Section \ref{sec:GKW.dec} we collect some notions about stochastic analysis and discuss semi-static variance optimal strategies. 
Section \ref{sec:sem.quad} is devoted to the study of the GKW decomposition for some square integrable martingales, which play a fundamental role in our setting.
In Section \ref{sec:GKW.eu} we discuss certain problems related to stochastic analysis of processes depending on a parameter and Fourier methods in a general multidimensional factor model. We also recover some of the results of \cite{KalP07} for affine stochastic volatility models. In Section \ref{sec:app} we apply the results of the previous sections to the Heston model (which is an affine model) and the $3/2$-model (which is non-affine).

\section{Basic Tools and Motivation}\label{sec:GKW.dec}
By $(\Om,\Fscr,\Pbb)$ we denote a complete probability space and by $\Fbb$ a filtration satisfying the usual conditions. We fix a time interval $[0,T]$, $T>0$, assume that $\Fscr_0$ is trivial and set $\Fscr=\Fscr_T$.

Because of the usual conditions of $\Fbb$, we only consider \cadlag martingales and when we say that a process $X$ is a martingale we implicitly assume the \cadlag property of the paths of $X$.

A real-valued martingale $X$ is square integrable if $X_T\in L\p2(\Pbb):=L\p2(\Om,\Fscr,\Pbb)$. By $\Hscr\p{\,2}=\Hscr\p{\,2}(\Fbb)$ we denote the set of real-valued $\Fbb$-adapted square integrable martingales. For $X\in\Hscr\p{\,2}$, we define $\|X\|_{\Hscr\p{\,2}}:=\|X_T\|_2$, where $\|\cdot\|_2$ denotes the $L\p2(\Pbb)$-norm. The space $(\Hscr\p{\,2},\|\cdot\|_2)$ is a Hilbert space; we also introduce $\Hscr\p{\,2}_{\,0}:=\{X\in\Hscr\p{\,2}: X_{\,0}=0\}$.

If $X,Y$ belong to $\Hscr\p{\,2}$, then there exists a unique predictable c\`adl\`ag process of finite variation, denoted by $\aPP{X}{Y}$ and called predictable covariation of $X$ and $Y$, such that $\aPP{X}{Y}_{\,0}=0$ and $XY-\aPP{X}{Y}$ is a uniformly integrable martingale. Clearly $\Ebb[X_TY_T-X_{\,0}Y_{\,0}]=\Ebb[\aPP{X}{Y}_T]$. We say that two local martingales $X$ and $Y$ are orthogonal if $XY$ is a local martingale starting at zero. If $X$ and $Y$ are (locally) square integrable (local) martingales, they are orthogonal if and only if $X_0Y_0=0$ and $\aPP{X}{Y}=0$. 

If $H$ is a measurable process and $A$ a process of finite variation, by $H\cdot A$ we denote the (Riemann--Stieltjes) integral process of $H$ with respect to $A$, i.\ e., $H\cdot A_t(\om):=\int_0\p tH_s(\om)\rmd A_s(\om)$. We also use the notation $\int_0\p\cdot H_s\rmd A_s$ to denote the process $H\cdot A$. We recall that $H\cdot A$ is of finite variation if and only if $|H|\cdot \mathrm{Var}(A)_t(\om)<+\infty$, for every $t\in[0,T]$ and $\om\in\Om$. Notice that, if $H\cdot A$ is of finite variation, then
\begin{equation}\label{eq:var.int}
\Var(H\cdot A)=|H|\cdot\Var(A).
\end{equation} 

For $X\in\Hscr\p{\,2}$, we define
\[
\Lrm_\Cbb\p2(X):=\{\vartheta\textnormal{ predictable and complex-valued: } \Ebb[|\vartheta|\p2\,\cdot\aPP{X}{X}_T]<+\infty\},
\] 
the space of complex-valued integrands for $X$. For $\vartheta\in\Lrm_\Cbb\p2(X)$, $\vartheta\cdot X$ denotes the stochastic integral process of $\vartheta$ with respect to $X$ and it is characterized as it follows: Let $Z$ be a complex-valued square integrable martingale. Then $Z=\vartheta\cdot X$ if and only if $Z_{\,0}=0$ and $\aPP{Z}{Y}=\vartheta\cdot\aPP{X}{Y}$, for every $Y\in\Hscr\p2$. We also use the notation $\int_0\p\cdot \vartheta_s\rmd X_s$ to indicate the martingale $\vartheta\cdot X$.
By $\Lrm\p2(X)$ we denote the subspace of predictable integrands in $\Lrm_\Cbb\p2(X)$ which are \emph{real-valued}.

\paragraph{Variance-Optimal Hedging. } In a financial market, where the price process is described by a strictly positive square integrable martingale $S$, a square integrable contingent claim $\et$ is given and $H=(\Ebb[\et|\Fscr_t])_{t\in[0,T]}$ is the associated martingale. We then consider the optimization problem
\begin{equation}\label{var.op.cl}
\ep\p2=\min_{c\in\Rbb,\vartheta\in\Lrm\p2(S)}\Ebb\big[\big(c+\vartheta\cdot S_T-\et\big)\p2\big].
\end{equation}
The meaning of \eqref{var.op.cl} is to minimize the variance of the hedging error: If the market is complete the solution of  \eqref{var.op.cl} is identically equal to zero, as perfect replication is possible. If the market is incomplete, the squared hedging error will be strictly positive, in general. In \cite{FS86} was shown that the solution $(c\p\ast,\vartheta\p\ast)$ of \eqref{var.op.cl} always exists and is given by the so-called Galtchouk-Kunita-Watanabe (GKW) decomposition of $H$. The GKW decomposition is the unique decomposition
\begin{equation}\label{eq:gkw.dec}
H=H_{\,0}+\vartheta\cdot X+L\,,
\end{equation}
where $\vartheta\in\Lrm\p2(X)$ and the martingale $L$ is in the orthogonal complement of $\Lscr\p2(X)$  in $(\Hscr\p{\,2}_{\,0}, \|\cdot\|_2)$. The solution of the hedging problem \eqref{var.op.cl} is then given by $c\p\ast=\Ebb[\et]$ and $\vartheta\p\ast$ being the integrand in the GKW decomposition of $H$ with respect to $S$. The minimal hedging error can be expressed as
\[
\ep\p2=\Ebb[\aPP{L}{L}_T]=\Ebb[L_T\p2],
\]
using the orthogonal component $L$ in the GKW decomposition \eqref{eq:gkw.dec}.

Moreover, notice that, if $L\in\Hscr\p2_0$ is orthogonal to $\Lscr\p2(S)$ in the Hilbert space sense, then $L$ is also orthogonal to $S$, i.\ e.\ $LS$ is a martingale starting at zero and, in particular, $\aPP{L}{S}=0$. Therefore, from \eqref{eq:gkw.dec} we can compute the optimal strategy $\vartheta^\ast$ by
\begin{equation}\label{eq:gkw.dec.int}
\aPP{S}{H}=\vartheta^\ast\cdot\aPP{S}{S}=\int_0\p\cdot\vartheta^\ast_s\,\rmd \aPP{S}{S}_s\,,
\end{equation}
that is $\vartheta^\ast=\rmd\aPP{S}{H}\slash\rmd\aPP{S}{S}$ in the Radon-Nikodym-derivative sense.

\paragraph{Semi-Static Variance-Optimal Hedging. } In \cite{DTHKR17} we have introduced the following generalization of the classic variance-optimal hedging problem \eqref{var.op.cl}: Assume that in addition to a dynamic (i.\ e.\@ continuously rebalanced) position in the underlying stock, a static (i.\ e.\@ buy-and-hold) position in a fixed basket of contingent claims $(\et^1, \dots, \et^d)$ is allowed. 
More precisely, let $S$ be a strictly positive square integrable martingale modelling the evolution of the stock price and $\et=(\et^1, \dots, \et^d)^\top$ the \emph{fixed} vector of square integrable contingent claims.
\begin{definition}\label{def:sems.st} 
(i) A \emph{semi-static} strategy is a pair $(\vartheta,v)\in\Lrm\p2(S)\times\Rbb\p d$. A semi-static strategy of the form $(\vartheta,0)$ (resp., $(0,v)$) is called a \emph{dynamic} (resp., \emph{static}) strategy.

(ii) A semi-static variance-optimal hedging strategy for the square integrable contingent claim $\et\p0$ is a solution of the \emph{semi-static variance-optimal hedging problem} given by
\begin{equation}\label{eq:setup}
\ep\p2=\min_{v\in\Rbb\p{d}, \vartheta \in \Lrm^2(S) ,c\in\Rbb}\Ebb\left[\Big(c-v\p\top\Ebb[\et]+\vartheta\cdot S_T-(\et\p{\,0}-v\p\top \et)\Big)\p{\,2}\right].    
\end{equation}
\end{definition}
Comparing the solution of \eqref{eq:setup} with the one of \eqref{var.op.cl}, it is clear that the latter one will be smaller or equal than the first one, as the minimization problem in \eqref{eq:setup} is taken over a bigger set. Therefore, semi-static strategies allow for a reduction of the quadratic hedging error in comparison with classic dynamic hedging in the underlying. 

Following \cite{DTHKR17}, the semi-static hedging problem can be split into an inner and outer optimization problem, i.\ e.\ it can be written as
\begin{equation}\label{eq:inner_outer}
\begin{cases}
\epsilon^2(v) = \min_{\vartheta \in \Lrm^2(S), c \in \RR}\EE\left[\Big(c-v^\top\Ebb[{\et}]+\vartheta\cdot S_T - (\et^0-v^\top \et)\Big)^2\right], & \quad \text{(inner prob.)}\\    
\epsilon^2 = \min_{v \in \RR^d} \epsilon(v)^2. &\quad \text{(outer prob.)}\\
\end{cases}
\end{equation}
Notice that the inner optimization problem in \eqref{eq:inner_outer} is a classical variance-optimal hedging problem as in \eqref{var.op.cl}, while the outer problem becomes a finite dimensional quadratic optimization problem of the form  
\begin{equation}\label{eq:hed.err}
\varepsilon(v)\p{\,2} =A - 2 v\p\top B+ v^\top C v\,.
\end{equation}
As shown in \cite[Theorem 2.3]{DTHKR17} the coefficients of this problem can be written as
\begin{align}\label{not:var.cov}
A:=\Ebb[\aPP{L\p0}{L\p0}_T],\qquad B\p i:=\Ebb[\aPP{L\p0}{L\p i}_T],\qquad C\p{ij}:=\Ebb[\aPP{L\p i}{L\p j}_T]\,,\quad i,j=1,\ldots,d.
\end{align}
where $L^0$ and $L^i$ are the orthogonal parts of the GKW-decompositions of $H^0$ and $H^i$ respectively. Moreover, we can write
\begin{equation}\label{eq:pb.res}
\aPP{L\p i}{L\p j}=\aPP{H\p i}{H\p j}-\vartheta\p i\vartheta\p j\cdot\aPP{S}{S},\quad i,j=0,\ldots,d.
\end{equation} 
where the $\vartheta^i$ are the integrands in the respective GKW decompositions. If $C$ is invertible, then the optimal strategy $(v^\ast, \vartheta^\ast)$ for the semi-static variance-optimal hedging problem \eqref{eq:setup} is given by 
\[v^* = C^{-1} B, \qquad \vartheta^\ast = \vartheta^0 - (v^\ast)^\top (\vartheta^1, \dotsc, \vartheta^d).\]

Hence, to solve the semi-static variance-optimal hedging problem, it is necessary to compute all covariations of the residuals in the GKW decomposition of $\et\p i$, $i=1,\ldots,d$. 
This extends beyond the classic variance-optimal hedging problem considered in \cite{KalP07}, where it was enough to determine $A$ in order to compute the squared hedging error and the optimal strategy.

\section{The GKW-decomposition in factor models}\label{sec:sem.quad}
 To solve the inner minimization problem in the semi-static hedging problem \eqref{eq:setup}, it is necessary to compute the predictable covariations of the GKW-residuals in \eqref{eq:pb.res} and then their expectation to get $A$, $B$ and $C$ in \eqref{not:var.cov}. In this section, we consider the particular case of a so-called factor-model, where both the underlying $S$ and the claims $Y^i$ of interest only depend on the state of a finite-dimensional economic factor process $X = (X^1, \dotsc, X^n)$. More precisely, we assume that $X$ is a quasi-left-continuous locally square integrable semimartingale (cf.\ \cite[Definition II.2.27]{JS00}) with state space $(E,\Escr):=(\Rbb\p n,\Bscr(\Rbb\p n))$; the underlying stock is a locally square integrable local martingale\footnote{In order to compute the quantities in \eqref{not:var.cov}, we will need the stronger assumption that $S$ is a square integrable martingale. However, this assumption is not needed to compute \eqref{eq:pb.res}} given by $S = e^{X^1}$, and the claims of interest are of the form 
\begin{equation}\label{eq:Y_factor}
Y_t^i = f^i(t,X_t^1, \dotsc, X_t^n)
\end{equation}
where the $f^i$ are in $C^{1,2}(\Rbb_+\times\Rbb\p n)$, $i=1,\ldots,d$. To simplify notation we only consider two claims $Y^1$ and $Y^2$; the extension to more than two claims is straight-forward. 

To prepare for our results and their proofs, we recall that any local martingale $X$ can be decomposed in a unique way as $X=X_0+X\p c+X\p d$, where $X\p c$ is a continuous local martingale starting at zero, called \emph{the continuous part} of $X$, while $X\p d$, called \emph{the purely discontinuous part} of $X$, is a local martingale starting at zero which is orthogonal to all adapted and continuous local martingales (cf.\ \cite[Theorem I.4.18]{JS00}).  We also recall that a locally square integrable semimartingale $X$ is a special semimartingale and thus there exists a semimartingale decomposition (which is unique)
\begin{equation}\label{eq:sp.smart.dec}
X\p j=X\p j_0+M\p j+A\p j,\quad j=1,\ldots,n,
\end{equation}
where $A\p j$ is a predictable processes of finite variation and $M\p j$ a locally square integrable local martingale. Notice that $X\p {j,c}=M\p {j,c}$. Since $X$ is quasi-left-continuous, the predictable processes $A\p j$ are continuous and the local martingales $M\p j$ are quasi-left-continuous (cf.\ \cite[Corollary 8.9]{HWY92}).

We denote by  $\mu$ be the jump measure of $X$ and by $\nu$ its predictable compensator. In particular, $\mu$ is an integer valued random measure on $[0,T]\times E$, $T>0$. The space $\Gscr\p2_\textnormal{loc}(\mu)$ of the predictable integrands for the compensated jump measure $(\mu-\nu)$ is defined in \cite[Eq.~(3.62)]{J79}, as the space of real-valued processes $W$ on $\Rbb_+\times E$ which are $\Pscr\otimes\Escr$-measurable, $\Pscr$ denoting the predictable $\sig$-algebra on $\Om\times[0,T]$, such that the increasing process
\[
\sum_{s\leq\cdot}W(s,\om,\Delta X_s(\om))\p 21_{\{\Delta X_s(\om)\neq0\}}
\] 
is locally integrable. For $W\in\Gscr_\textnormal{loc}\p2(\mu)$, we denote by $W\ast(\mu-\nu)$ the stochastic integral of $W$ with respect to the compensated jump measure $(\mu-\nu)$ and define it as the unique purely discontinuous locally square integrable local martingale $Y$ such that
\[
\Delta Y_t(\om)=W(t,\om,\Delta X_t(\om)),
\] 
$\Delta Z$ denoting the jump process of a \cadlag process $Z$ with the convention $\Delta Z_0=0$.

The existence and uniqueness of this locally square integrable local martingale is guaranteed by \cite[Theorem I.4.56, Corollary I.4.19]{JS00}.
Observe that, if $U,W\in\Gscr\p2_\mathrm{loc}(\mu)$, then \cite[Theorem II.1.33(a)]{JS00}, yields
\begin{equation}\label{eq:com.pb}
\aPP{U\ast(\mu-\nu)}{W\ast(\mu-\nu)}_t=\int_{[0,t]\times E}U(s,x)W(s,x)\nu(\rmd s,\rmd x),\quad t\in[0,T].
\end{equation}

The next lemma is the key result for the computation of \eqref{eq:pb.res} under the factor-model assumption. 

\begin{lemma}\label{lem:rep.it.for}
Let $X$ be a locally square integrable semimartingale with canonical decomposition as in \eqref{eq:sp.smart.dec}. Let $Y=(Y_t)_{t\in[0,T]}$ be a locally square integrable local martingale of the form $Y_t=f(t,X\p1_t,\ldots,X\p n_t)$, where $f \in C^{1,2}(\Rbb_+\times\Rbb\p n)$. Then,
\begin{equation}\label{eq:rep.it.for}
Y=Y_0+\sum_{j=1}\p n\partial_{x_j} f_-\cdot X\p {j,c}+W\ast(\mu-\nu)
\end{equation}
where $f_t:=f(t,X\p1_t,\ldots,X\p n_t)$, $f_{t-}:=f(t,X\p1_{t-},\ldots,X\p n_{t-})$, $\partial_{x_j}$ denotes the $j$-th partial derivatives of $f$ and 
\begin{equation}\label{eq:proc.W}
W(t,\om,x_1,\ldots,x_n):=f(t,x_1+X\p1_{t-}(\om),\ldots,x_n+X\p n_{t-}(\om))-f(t,X\p1_{t-}(\om),\ldots,X\p n_{t-}(\om)).
\end{equation}
\end{lemma}
\begin{proof}
We denote by $\mathrm{Id}$ the identity on $[0,T]$. Applying It\^o formula to the $\Rbb\p{n+1}$-valued semimartingale $(\mathrm{Id},X\p1,\ldots,X\p n)$ (cf.\ \cite[Theorem I.4.57]{JS00}), from the canonical decomposition \eqref{eq:sp.smart.dec}, we deduce
\begin{equation}\label{eq:it.for.lem}
\begin{split}
Y=Y_0+&\sum_{j=1}\p n\partial_{x_j} f_-\cdot M\p j+\partial_t f_-\cdot \textnormal{Id}+\sum_{j=1}\p n\partial_{x_j} f_-\cdot A\p j+\frac12\sum_{j,k=1}\p n\partial_{x_jx_k}\p2 f_-\cdot\aPP{X\p {j,c}}{X\p {k,c}}\\&+\sum_{s\leq\cdot}\Delta Y_{s}-\sum_{j=1}\p n\partial_{x_j} f_{s-}\Delta X\p j_s,
\end{split}
\end{equation}
where $\partial\p2_{\cdot\cdot}$ denotes the second-partial-derivative operator.
We define $N:=\sum_{j=1}\p n\partial_{x_j} f_-\cdot M\p j$,
\[
A:=\partial_t f_-\cdot \textnormal{Id}+\sum_{j=1}\p n\partial_{x_j} f_-\cdot A\p j+\frac12\sum_{j,k=1}\p n\partial_{x_jx_k}\p2 f_-\cdot\aPP{X\p {j,c}}{X\p {k,c}}
\]
and finally
\[
\begin{split}
U(t,\om,x_1,x_2)&:=\\&f(t,x_1+X\p1_{t-}(\om),\ldots,x_n+X\p n_{t-}(\om))-f(t,X\p1_{t-}(\om),\ldots,X\p n_{t-}(\om))-\sum_{j=1}\p n\partial_{x_j} f_{t-} x\p j.
\end{split}
\]
The process $U$ belongs to $\Gscr\p2_\mathrm{loc}(\mu)$. Indeed, because of \eqref{eq:it.for.lem} and the continuity of $A$, $U$ is the jump process of the locally square integrable local martingale $Y-N$. Hence $\sum_{s\leq\cdot}U_s\p2\leq[Y-N,Y-N]$ and the right-hand side of this estimate is locally integrable because of \cite[Proposition I.4.50(b)]{JS00}. This means that we can define the locally square integrable local martingale $U\ast(\mu-\nu)$. The process $\sum_{s\leq t}U(s,\om,\Delta X\p1_s(\om),\ldots\Delta X\p n_s(\om))+A_t(\om)=U\ast\mu_t(\om)+A_t(\om)$ is of finite variation and by \eqref{eq:it.for.lem} it is a locally square integrable local martingale. Hence it is purely discontinuous and has the same jumps as $U\ast(\mu-\nu)$. From \cite[Corollary I.4.19]{JS00}, these two local martingales are indistinguishable. Hence \eqref{eq:it.for.lem} becomes
\[
Y=Y_0+\sum_{j=1}\p n\partial_{x_j} f_-\cdot M\p j+U\ast(\mu-\nu).
\]
From \cite[Theorem 11.24]{HWY92}, we can write $M\p {j,d}=x_j\ast(\mu-\nu)$, $M\p {j,d}$ denoting the purely discontinuous part of the local martingales $M\p j$, $j=1,\ldots,n$. Using that $M\p j=X\p{ j,c}+M\p{ j,d}$ and \cite[Proposition II.1.30(b)]{JS00}, we get by linearity of the involved stochastic integrals
\[
Y=Y_0+\sum_{j=1}\p n\partial_{x_j} f_-\cdot X\p{ j,c}+\left(U+\sum_{j=1}\p n\partial_{x_j} f_-x_j\right)\ast(\mu-\nu),
\]
$j=1,\ldots,n$, which is the desired result.
\end{proof}
Notice that \eqref{eq:rep.it.for} is the decomposition of $Y$ in its continuous and purely discontinuous martingale part $Y\p c$ and $Y\p d$ respectively. It can be rephrased saying that every locally square integrable local martingale $Y$ such that $Y_t=f(t,X_t\p1,\ldots,X_t\p n)$ can be represented as the sum of stochastic integrals with respect to the continuous martingale parts of $X$ and a stochastic integral with respect to the compensated jump measure of  $X$.

Applying Lemma \ref{lem:rep.it.for} to the locally square integrable local martingale $S=\rme\p{X\p1}$, we get
\begin{equation}\label{eq:exp.S}
S=S_0+S_-\cdot X\p{1,c}+\big(S_-(\rme\p{x_1}-1)\big)\ast(\mu-\nu).
\end{equation}
We observe that the processes $S$ and $S_-$ are strictly positive. Hence the process $S\p{-1}_-$ is locally bounded and $S\p{-1}_-\cdot S$ is a locally square integrable local martingale. Furthermore $\Delta \big[S\p{-1}_-\cdot S\big]=(\rme\p{\Delta X\p1}-1)$. Therefore the function $x_1\mapsto(\rme\p {x_1}-1)$ belongs to $\Gscr\p2_\textnormal{loc}(\mu)$ and we can define 

\begin{equation}\label{eq:def.R}
\widetilde S\p d:=g\ast(\mu-\nu),\qquad g(x_1):=(\rme\p {x_1}-1),
\end{equation} 
which is a locally square integrable local martingale. Combining this with \eqref{eq:exp.S} and \cite[Proposition II.1.30(b)]{JS00}, we get
\begin{equation}\label{eq:exp.S.st.ex}
S=S_0+S_-\cdot \widetilde S,\qquad \widetilde S:=(X\p{1, c}+\widetilde S\p d).
\end{equation}
Notice that $\widetilde S$ is the s\emph{tochastic logarithm} of $S$ (cf.\ \cite[Theorem II.8.10]{JS00}). Moreover, if $Y$ is as in Lemma \ref{lem:rep.it.for}, from \eqref{eq:exp.S.st.ex} and \eqref{eq:com.pb}, we have
\begin{equation}\label{eq:pb.Y.S}
\aPP{S}{Y}=\sum_{j=1}\p n(S_-\partial_{x_j} f_-)\cdot\aPP{X\p {j,c}}{X\p{1, c}}+\int_{[0,\cdot]\times E}(\rme\p{x_1}-1)S_{s-}W(s,x)\nu(\rmd s,\rmd x).
\end{equation}

The next result is our main result on the predictable covariation of the GKW-residuals under the factor-model assumption \eqref{eq:Y_factor}.

\begin{theorem}\label{thm:poi.br.res.smar}
Let $X=(X\p1,\ldots,X\p n)$ be a locally square integrable semimartingale. Let $Y\p1$ and $Y\p2$ be locally square integrable local martingales such that $Y\p i_t=f\p i(t,X\p1_t,\ldots,X\p n_t)$, where $f\p i$ satisfies the assumptions of Lemma \ref{lem:rep.it.for}, and let $Y\p i=Y\p i_0+\vartheta\p i\cdot S+L\p i$ be the GKW decomposition of $Y\p i$ with respect to $S$, $i=1,2$. Then the explicit form of the integrand $\vartheta\p i$ is
\begin{equation}\label{eq:var.opt.str}
\vartheta_t\p i=\frac{1}{S_{t-}}\left(\sum_{j=1}\p n\partial_{x_j}f\p i_{t-}\frac{\rmd\aPP{X\p{1,c}}{X\p{ j,c}}_t}{\rmd\aPP{\widetilde S}{\widetilde S}_t}+\frac{\rmd\aPP{\widetilde S\p d}{Y\p {i,d}}_t}{\rmd\aPP{\widetilde S}{\widetilde S}_t}\right),\qquad i=1,2.
\end{equation}
Furthermore the differential of the predictable covariation of the residuals $L\p 1$ and $L\p 2$ in the GKW decomposition of $Y\p1$ and $Y\p2$ with respect to $S$ is
\begin{equation}\label{eq:res.pb.exp}
\begin{split}
\rmd \aPP{L\p 1}{L\p 2}_t&=\rmd\aPP{Y\p1}{Y\p2}_t-\vartheta\p 1_t\vartheta\p 2_t\rmd\aPP{S}{S}_t
\\&=
\sum_{j=1}\p n\sum_{k=1}\p n\partial_{x_j}f\p1_{t-}\partial_{x_k}f\p2_{t-}\left(\rmd\aPP{X\p{j, c}}{X\p{k, c}}_t-\frac{\rmd\aPP{X\p{1, c}}{X\p{j, c}}_t}{\rmd\aPP{\widetilde S}{\widetilde S}_t}\rmd\aPP{X\p{1, c}}{X\p{k, c}}_t\right)
\\&-
\sum_{\begin{subarray}{c}i,\ell=1,\\ i\neq\ell\end{subarray}}\p2\sum_{j=1}\p n\partial_{x_j}f\p i_{t-}\frac{\rmd\aPP{\widetilde S\p d}{Y\p{\ell,d}}_t}{\rmd\aPP{\widetilde S}{\widetilde S}_t}\rmd\aPP{X\p{1, c}}{X\p{j, c}}_t
\\&+
\rmd\aPP{Y\p {1, d}}{Y\p{ 2, d}}_t-\frac{\rmd\aPP{\widetilde S\p d}{Y\p {1, d}}_t}{\rmd\aPP{\widetilde S}{\widetilde S}_t}\rmd\aPP{\widetilde S\p d}{Y\p {2, d}}_t
\end{split}
\end{equation}
where,
\begin{eqnarray}
\aPP{\widetilde S}{\widetilde S}&=&\aPP{X\p{1, c}}{X\p{1, c}}+\int_{[0,\cdot]\times E}(\rme\p{x_1}-1)\p 2\nu(\rmd s,\rmd x),\label{eq:pb.St.St}
\\
\aPP{\widetilde S\p d}{\widetilde S\p d}&=&\int_{[0,\cdot]\times E}(\rme\p{x_1}-1)\p 2\nu(\rmd s,\rmd x),\label{eq:pb.R.R}\\
\aPP{\widetilde S\p d}{Y\p {i, d}}&=&\int_{[0,\cdot]\times E}(\rme\p{x_1}-1)W\p i(s,x)\nu(\rmd s,\rmd x)\label{eq:pb.R.Yi},\\
 \aPP{Y\p {1, d}}{Y\p {2, d}}&=&\int_{[0,\cdot]\times E}W\p 1(s,x)W\p 2(s,x)\nu(\rmd s,\rmd x) \label{eq:pb.Y1.Y2}.
\end{eqnarray}
and $W\p i$ is given by \eqref{eq:proc.W} with $f=f\p i$, $i=1,2$. 
\end{theorem}
\begin{proof}
We first show \eqref{eq:var.opt.str}. From \eqref{eq:gkw.dec.int}, we have
\begin{equation}\label{eq:thti}
\vartheta\p i=\frac{\rmd\aPP{S}{Y\p i}_t}{\rmd\aPP{S}{S}_t}.
\end{equation}
From \eqref{eq:exp.S.st.ex}, $\aPP{S}{S}=S_-\p2\cdot\aPP{\widetilde S}{\widetilde S}$. To compute $\aPP{\widetilde S}{\widetilde S}$, we use \eqref{eq:def.R}, \eqref{eq:exp.S.st.ex} and \eqref{eq:com.pb}. This, in particular shows \eqref{eq:pb.R.R} and \eqref{eq:pb.St.St}. The process $\aPP{S}{Y\p i}$ is given by \eqref{eq:pb.Y.S} with $Y=Y\p i$  and $W=W\p i$. Inserting this expression for $\aPP{S}{Y\p i}$ and the previous one for $\aPP{S}{S}$ in \eqref{eq:thti}, yields \eqref{eq:var.opt.str}. To compute $\aPP{Y\p1}{Y\p 2}$ we again use Lemma \ref{lem:rep.it.for} and \eqref{eq:com.pb}: The computations are similar to those for the computation of $\vartheta\p i$. The proof of \eqref{eq:res.pb.exp} is straightforward, once $\aPP{Y\p1}{Y\p 2}$, $\aPP{S}{S}$, $\vartheta\p1$ and $\vartheta\p 2$ are known. The proof of the theorem is now complete.
\end{proof}

If the semimartingale $X$ is continuous, then the formulas in Theorem \ref{thm:poi.br.res.smar} become simpler. Indeed, in this case, all purely discontinuous martingales appearing in \eqref{eq:res.pb.exp} vanish and the following corollary holds:
\begin{corollary}\label{cor:poi.br.res.smar.con}
Let the assumptions of Theorem \ref{thm:poi.br.res.smar} be satisfied and furthermore assume that the semimartingale $X$ is continuous. Then $\aPP{\widetilde S}{\widetilde S}=\aPP{X\p{1, c}}{X\p{1, c}}$ and \eqref{eq:var.opt.str} becomes

\begin{equation}\label{eq:var.opt.str.con}
\vartheta_t\p i=\frac{1}{S_{t}}\sum_{j=1}\p n\partial_{x_j}f\p i_{t}\frac{\rmd\aPP{X\p{1,c}}{X\p{ j,c}}_t}{\rmd\aPP{X\p{1, c}}{X\p{1, c}}_t}
\end{equation}
while \eqref{eq:res.pb.exp} reduces to
\begin{equation}\label{eq:res.pb.exp.con}
\rmd \aPP{L\p1}{L\p2}_t=\sum_{j=2}\p n\sum_{k=2}\p n\partial_{x_j}f\p1_{t}\partial_{x_k}f\p2_{t}\left(\rmd\aPP{X\p{j, c}}{X\p{k, c}}_t-\frac{\rmd\aPP{X\p{1, c}}{X\p{j, c}}_t}{\rmd\aPP{X\p{1, c}}{X\p{1, c}}_t}\rmd\aPP{X\p{1, c}}{X\p{k, c}}_t\right).
\end{equation}
\end{corollary}
Notice that the summations in \eqref{eq:res.pb.exp.con} start from $j=2$ and $k=2$.

We now denote by $(B,C,\nu)$ the semimartingale characteristics of $X$ (cf.\ \cite[II.\S2a]{JS00}) with respect to the truncation function $h=(h\p1,\ldots,h\p n)$, where $h\p j(x_j)=x_j1_{\{|x_j|\leq1\}}$, $j=1,\ldots,n$, and assume that they are absolutely continuous with respect to the Lebesgue measure, that is
\begin{equation}\label{eq:sm.ch.ac}
B_t=\int_0\p tb_s\rmd s,\quad C_t=\int_0\p tc_s\rmd s,\quad \nu(\om,\rmd t,\rmd x)=K_t(\om,\rmd x)\rmd t
\end{equation}  
where $b$ and $c$ are predictable processes taking values in $\Rbb\p n$ and in the subspace of symmetric matrices in $\Rbb\p{n\times n}$ respectively, while $K_t(\om,\rmd x)$ is a predictable kernel as in \cite[Proposition II.2.9]{JS00}. We stress that for every $t$ and $\om$, $K_t(\om,\rmd x)$ is a measure on $(\Rbb\p n,\Bscr(\Rbb\p n))$. Furthermore, we have
\[
c\p{jk}_t=\rmd\aPP{X\p{j, c}}{X\p{k, c}}_t, \quad j,k=1,\ldots,n.
\]
The triplet $(b,c,K)$ is sometimes called \emph{differential characteristics} of the semimartingale $X$.
In this context we get the following corollary to Theorem \ref{thm:poi.br.res.smar}:
\begin{corollary}\label{cor:poi.br.res.smar.abs.con}
Let the assumptions of Theorem \ref{thm:poi.br.res.smar} hold and furthermore let the semimartingale characteristics $(B,C,\nu)$ of $X$ be absolutely continuous with respect to the Lebesgue measure on $[0,T]$ and let $(b,c,K)$ be the corresponding differential characteristics. 

\textnormal{(i)} Then $\vartheta\p i$, $i=1,2$, becomes
\begin{equation}\label{eq:var.opt.str.abs.con}
\vartheta_t\p i=\frac{1}{S_{t-}\xi_t}\left(\sum_{j=1}\p n\partial_{x_j}f\p i_{t-}c\p{1j}_t+\int_E(\rme\p{x_1}-1)W\p i(t,x)K_t(\rmd x)\right).
\end{equation}
while \eqref{eq:res.pb.exp} reduces to
\begin{equation}\label{eq:res.pb.exp.abs.co}
\begin{split}
\frac{\rmd \aPP{L\p 1}{L\p 2}_t}{\rmd t}&
=\sum_{j=1}\p n\sum_{k=1}\p n\partial_{x_j}f\p1_{t-}\partial_{x_k}f\p2_{t-}\left(c\p{jk}_t-\frac{c\p{1j}_t}{\xi_t}c\p{1k}_t\right)
\\&-
\frac1{\xi_t}\sum_{\begin{subarray}{c}i,l=1,\\ i\neq\ell\end{subarray}}\p2\sum_{j=1}\p n\partial_{x_j}f\p i_{t-}c\p{j,1}_t\int_E(\rme\p{x_1}-1)W\p \ell(t,x)K_t(\rmd x)
\\&+
\int_EW\p1(t,x)W\p2(t,x)K_t(\rmd x)
\\&-\frac1{\xi_t}\left(\int_E(\rme\p{x_1}-1)W\p1(t,x)K_t(\rmd x)\right)\left(\int_E(\rme\p{x_1}-1)W\p2(t,x)K_t(\rmd x)\right)
\end{split}
\end{equation}
 where, $\xi_t:=c_t\p{11}+\int_E(\rme\p{x_1}-1)\p2K_t(\rmd x)$, $t\in[0,T]$.

\textnormal{(ii)} If furthermore $X$ is continuous then
\begin{equation}\label{eq:var.opt.str.abs.con.con}
\vartheta_t\p i=\frac{1}{S_{t}c\p{11}_t}\sum_{j=1}\p n\partial_{x_j}f\p i_{t}c\p{1j}_t
\end{equation}
and
\begin{equation}\label{eq:res.pb.exp.abs.con}
\rmd \aPP{L\p 1}{L\p 2}_t
=\sum_{j=2}\p n\sum_{k=2}\p n\partial_{x_j}f\p1_{t}\partial_{x_k}f\p2_{t}\left(c\p{jk}_t-\frac{c\p{1j}_t}{c\p{11}_t}c\p{1k}_t\right)\rmd t
\end{equation}
\end{corollary}

\section{Models with Fourier representation}\label{sec:GKW.eu}
In this section we combine variance-optimal hedging with Fourier methods. We do not assume a special stochastic volatility model, as in e.\ g.\  \cite{KalP07} where affine stochastic volatility models were considered. We rather work in a general multidimensional factor-model setting. This requires some technical considerations about stochastic processes depending on a parameter, which we discuss in Subsection \ref{subsec:par.proc}. In Subsection \ref{subsec:four.met} we consider contingent claims whose pay-offs have a Fourier representation. As a special case, we discuss semimartingale stochastic volatility models and, in particular, affine models recovering some of the results of \cite{KalP07}.

As a preliminary, we discuss the notion of `variation' for complex-valued processes. If $C=C\p 1+\rmi C\p 2$ is a complex-valued process and $C\p1$ and $C\p2$ are its real and imaginary part, respectively, we set 
\begin{equation}\label{eq:var.com.val.proc}
\Var(C):=\Var(C\p1)+\Var(C\p2)\,.
\end{equation}
In particular, from $|C|\leq|C\p1|+|C\p2|$, \eqref{eq:var.com.val.proc} yields $|C|\leq\Var(C)$.
Let $A$ be a real-valued process of finite variation and $K=K\p1+\rmi K\p2$ a measurable complex-valued process. Then \eqref{eq:var.com.val.proc} and \eqref
{eq:var.int} imply
\begin{equation}\label{eq:var.int.com}
\Var(K\cdot A)=(|K\p1|+|K\p2|)\cdot\Var(A).
\end{equation}

A complex-valued process $Z=X+\rmi Y$ is a square integrable martingale if the real-valued processes $X$ and $Y$ are square integrable martingales. For two complex-valued square integrable martingales $Z\p1=X\p1+\rmi Y\p1$ and $Z\p2=X\p2+\rmi Y\p2$ we define
\[
\aPP{Z\p1}{Z\p2}:=\big(\aPP{X\p1}{X\p2}-\aPP{Y\p1}{Y\p2}\big)+\rmi\big(\aPP{X\p1}{Y\p2}+\aPP{Y\p1}{X\p2}\big)
\] and the variation process $\Var(\aPP{Z\p1}{Z\p2})$ is given by \eqref{eq:var.com.val.proc}. Notice that $Z\p1Z\p2-\aPP{Z\p1}{Z\p2}$ is a martingale and $\aPP{Z\p1}{Z\p2}$ is the unique complex-valued predictable process of finite variation starting at zero with this property. Furthermore, because of Kunita--Watanabe inequality for real-valued martingales (cf.\ \cite[Corollary II.22, p25]{Me76}) we immediately get
\begin{equation}\label{eq:kwi.com}
\Ebb\big[\Var(\aPP{Z\p1}{Z\p2})_t\big]\leq2\Ebb\big[\aPP{Z\p1}{\overline Z\p1}_t\big]\p{1/2}\Ebb\big[\aPP{Z\p2}{\overline Z\p2}_t\big]\p{1/2}\,,
\end{equation} 
where $\overline Z\p j$ denotes the complex conjugate of $Z\p j$, $j=1,2$.

Notice that, if $Z\p1$ and $Z\p 2$ are complex-valued square integrable martingales, then the GKW decomposition of $Z\p2$ with respect to $Z\p1$ clearly holds but the residual is a complex-valued square integrable martingale and the integrand belongs to $\Lrm\p2_\Cbb(Z\p 1)$. Obviously also an analogous relation as \eqref{eq:gkw.dec.int} holds.

\subsection{Complex-Valued Processes Depending on a Parameter}\label{subsec:par.proc}
 We denote by $\Sscr$ the space of parameters and assume that it is a Borel subspace of $\Cbb\p n$ and $\Bscr(\Sscr)$ denotes the Borel $\sig$-algebra on $\Sscr$. On $(\Sscr,\Bscr(\Sscr))$ a finite complex measure $\zt$ is given. We recall that with $\zt$ we can associate the \emph{positive} measure $|\zt|$ called \emph{total variation} of $\zt$. Furthermore, there exists a complex-valued function $h$ such that $|h|=1$ and $\rmd\zt=h\,\rmd|\zt|$, so $L\p1(\zt)=L\p1(|\zt|)$. For details about complex-valued measures see \cite[\S6.1]{Ru87}. Note that Fubini's theorem also holds for products of complex-valued measures (cf.\ \cite[Theorem 8.10.3]{C13}).
 
\begin{definition}\label{def:joi.mea}
 Let $U(z)=(U(z)_t)_{t\in[0,T]}$ be a (complex-valued) stochastic process for every $z\in\Sscr$. By $U(\cdot)$ we denote the mapping $(t,\om,z)\mapsto U(\om,z)_t$.

\textnormal{(i)} We say that $U(\cdot)$ is jointly measurable if it is a $\Bscr([0,T])\otimes\Fscr_T\otimes\Bscr(\Sscr)$-measurable mapping. 

\textnormal{(ii)} We say that $U(\cdot)$ is \emph{jointly progressively measurable} if for every $s\leq t$ the mapping $(s,\om,z)\mapsto U(\om,z)_s$ is $\Bscr([0,t])\otimes\Fscr_t\otimes\Bscr(\Sscr)$-measurable. 

\textnormal{(iii)} We say that $U(\cdot)$ is \emph{jointly optional} if it is $\Oscr\otimes\Bscr(\Sscr)$-measurable, $\Oscr$ denoting the optional $\sig$-algebra on $[0,T]\times\Om$. 

\textnormal{(iv)} We say that $U(\cdot)$ is \emph{jointly predictable} if it is $\Pscr\otimes\Bscr(\Sscr)$-measurable, $\Pscr$ denoting the predictable $\sig$-algebra on $[0,T]\times\Om$.
\end{definition}

Let $U(z)$ be a stochastic process for every $z\in\Sscr$. If there exists a jointly measurable (or jointly progressively measurable, or jointly optional or jointly predictable) mapping $(t,\om,z)\mapsto Y(\om,z)_t$ such that, for every $z\in\Sscr$ the processes $U(z)$ and $Y(z)$ are indistinguishable, we identify them in notation, i.e.\@ $U(\cdot):=Y(\cdot)$. We shall use this convention without further mention.

\begin{proposition}\label{prop:poi.br.com}
Let $U(z)$ be a square integrable martingale for every $z\in\Sscr$ and let the estimate $\sup_{z\in\Sscr}\Ebb[|U(z)_T|\p2]<+\infty$ hold.

\textnormal{(i)} If $z\mapsto U(\om,z)_t$ is $\Bscr(\Sscr)$-measurable for every $t\in[0,T]$, the process $U=(U_t)_{t\in[0,T]}$,
\begin{equation}\label{eq:proc.U}
U_t:=\int_\Sscr U(z)_t\zt(\rmd  z),\quad t\in[0,T], 
\end{equation}
 is a square integrable martingale, provided that it is adapted.

\textnormal{(ii)} For every complex-valued square integrable martingale $Z$ such that $(t,\om,z)\mapsto\aPP{Z}{U(z)}_t(\om)$ is jointly measurable, the mapping $(t,\om,z)\mapsto \Var(\aPP{Z}{U(z)})_t(\om)$ is jointly measurable as well and
\begin{equation}\label{eq:est.var} 
\Ebb\left[\int_\Sscr \Var(\aPP{Z}{U(z)})_t|\zt|(\rmd z)\right]\leq2\Ebb\big[|Z_T|\p2\big]\p{1/2}\left(\sup_{z\in\Sscr} \Ebb\big[|U(z)_T|\p2\big]\right)\p{1/2}|\zt|(\Sscr)<+\infty.
\end{equation}

 If furthermore the process $D:=\int_\Sscr\aPP{Z}{U(z)}\zt(\rmd  z)$ is predictable, then
\begin{equation}\label{eq:com.ZU}
\aPP{Z}{U}=\int_\Sscr\aPP{Z}{U(z)}\zt(\rmd  z),
\end{equation}
whenever the process $U$ defined in \eqref{eq:proc.U} is a square integrable martingale.

\textnormal{(iii)} If there exists a complex-valued jointly progressively measurable process $K(\cdot)$ such that
\[
\aPP{Z}{U(z)}_t=\int_0\p t K(z)_s\rmd s\,,
\]
then the identity
\begin{equation}\label{eq:ch.in}
D=\int_0\p\cdot\int_\Sscr K(z)_s \zt(\rmd z)\rmd s
\end{equation}
holds. Hence $D$ is predictable and
\begin{equation}\label{eq:est.pb.ass.con.as}
\aPP{Z}{U}=\int_0\p \cdot\int_\Sscr K(z)_s\zt(\rmd z)\rmd s,
\end{equation}
whenever $U$ is well-defined and adapted.
\end{proposition}
\begin{proof}
To see (i) we assume that $U$ is well-defined and adapted and verify the martingale property and the square integrability. From
\[
\Ebb\big[|U_t|\p2\big]\leq\left(\int_\Sscr\Ebb\big[|U(z)_t|\p2\big]|\zt|(\rmd z)\right)\p{1/2}\Big(|\zt|(\Sscr)\Big)\p{1/2}\leq\left(\sup_{z\in\Sscr} \Ebb\big[|U(z)_T|\p2\big]\right)\p{1/2}|\zt|(\Sscr)<+\infty
\]
we deduce the square integrability of $U$. Furthermore we can apply Fubini's theorem to get
\[
\Ebb[U_t|\Fscr_s]=\int_\Sscr \Ebb\left[U(z)_t|\Fscr_s\right]\zeta(\rmd z)=U_s\quad \textnormal{a.s.},\quad 0\leq s\leq t\leq T,
\]
 and this completes the proof of (i).
We now verify (ii). The joint measurability of $\Var(\aPP{Z}{U(\cdot)})$ follows from the joint measurability of $\aPP{Z}{U(\cdot)}$ and from the definition of variation process. To prove \eqref{eq:est.var} observe that Tonelli's theorem, Kunita--Watanabe inequality \eqref{eq:kwi.com} and $\Ebb[\aPP{X}{\overline X}_t]\leq\Ebb[|X\p2_t|]$, $t\geq0$, for every complex-valued square integrable martingale $X$, imply
\[\begin{split}
\Ebb\left[\int_\Sscr \Var(\aPP{Z}{U(z)})_t|\zt|(\rmd z)\right]&=\int_\Sscr\Ebb\big[\Var(\aPP{Z}{U(z)})_t\big]|\zt|(\rmd z)
\\&\leq\sup_{z\in\Sscr} \Ebb\big[\Var(\aPP{Z}{U(z)})_t\big]|\zt|(\Sscr)
\\&\leq
2\Ebb\big[|Z_T|\p2\big]\p{1/2}\left(\sup_{z\in\Sscr} \Ebb\big[|U(z)_T|\p2\big]\right)\p{1/2}|\zt|(\Sscr)<+\infty.
\end{split}
\]
 This proves \eqref{eq:est.var}. To see \eqref{eq:com.ZU}, because of \eqref{eq:est.var} we can apply Fubini's theorem. From (i), for $ 0\leq s\leq t\leq T$, we compute
\[
\begin{split}
\Ebb[Z_tU_t-D_t|\Fscr_s]&=\int_\Sscr\Ebb[Z_tU(z)_t-\aPP{Z}{U(z)}_t|\Fscr_s]\zt(\rmd z)\\&=\int_\Sscr \big(Z_sU(z)_s-\aPP{Z}{U(z)}_s\big)\zt(\rmd z)\\&=Z_sU_s-D_s,
\end{split}
\]
which is \eqref{eq:com.ZU} because $D$ is a predictable process of finite variation starting at zero such that $ZU-D$ is a martingale. Finally, we show (iii). First we notice that the mapping $(s,\om)\mapsto\int_\Sscr K(\om,z)_s\zt(\rmd z)$ is $\Bscr([0,t])\otimes\Fscr_t$-measurable,  for $s\leq t$, that is, it is a progressively measurable process. Therefore, the stochastic process $\int_0\p \cdot\int_\Sscr K(z)_s\zt(\rmd z)\rmd s$ is adapted and continuous, hence predictable. Furthermore, the mapping $(\om,z)\mapsto\int_0\p t K(z)_s \rmd s$ is $\Fscr_t\otimes\Bscr(\Sscr)$-measurable. Therefore $D=\int_\Sscr\int_0\p\cdot K(z)_s \rmd s \zt(\rmd z)$ is an adapted process.
We now observe that, because of \eqref{eq:var.int.com} and (ii), the estimation
\begin{equation}\label{eq:est.Y}
\int_0\p t |K(z)_s|\rmd s\leq \Var(\aPP{Z}{U(z)})_T
\end{equation}
holds. From \eqref{eq:est.var} and \eqref{eq:est.Y}, we then have
\[
\Ebb\left[\int_\Sscr\int_0\p t |K(z)_s|\rmd s|\zt|(\rmd z)\right]\leq \Ebb\left[\int_\Sscr \Var(\aPP{Z}{U(z)})_T|\zt|(\rmd z)\right]<+\infty.
\]
Hence, applying Fubini's theorem we deduce 
\begin{equation}\label{eq:pred.con.mod}
\int_\Sscr\int_0\p t K(z)_s\rmd s\zt(\rmd z)=\int_0\p t\int_\Sscr K(z)_s\zt(\rmd z)\rmd s,\quad \textnormal{a.s., }\quad t\in[0,T].
\end{equation}
By \eqref{eq:est.Y}, because from \eqref{eq:est.var} the mapping $(\om,z)\mapsto \Var(\aPP{Z}{U(z)})_T(\om)$ belongs to $L\p1(|\zt|)$ a.\ s., an application of Lebesgue's theorem on dominated convergence now yields that the left-hand side of \eqref{eq:pred.con.mod} is a.\ s.\ continuous. Hence, identifying $D$ with a continuous version, we can claim the $D$ and $\int_0\p \cdot\int_\Sscr K(z)_s\zt(\rmd z)\rmd s$ are indistinguishable. In particular, the process $D$ of (ii) is predictable, because it is continuous and adapted. From (ii), we get \eqref{eq:est.pb.ass.con.as}. The proof of the proposition is now complete.
\end{proof}

We remark that a sufficient condition for the process $U$ in Proposition \ref{prop:poi.br.com} to be well defined and adapted is the joint progressive measurability of the mapping $U(\cdot)$.

We conclude this subsection with the following lemma:
\begin{lemma}\label{lem:joi.pred.int.fv}
Let $K(\cdot)$ be a jointly predictable complex-valued mapping and $A$ a predictable increasing process. Let $\int\p T|K(\om,z)_s|\rmd A_s(\om)<+\infty$ a.\ s. Then the mapping $(t,\om,z)\mapsto\int_0\p t K(\om,z)_s\rmd A_s(\om)$ is jointly predictable.
\end{lemma}
\begin{proof}
Let $f:\Sscr\longrightarrow\Rbb$ be a $\Bscr(\Sscr)$-measurable bounded real-valued function and let $K(\cdot)$ be of the form $K(t,\om,z)=f(z)k_t$, where $k$ is a bounded real-valued predictable process. Let $\Cscr$ denote the class of all real-valued predictable processes of this form. Then, by \cite[Proposition I.3.5]{JS00}, for any $K(\cdot)\in\Cscr$, the mapping $(t,\om,z)\mapsto\int_0\p tK(z,\om)_s\rmd A_s(\om)$ is jointly predictable.  If now $K(\cdot)$ is real-valued and bounded, by the monotone class theorem (see \cite[Theorem 1.4]{HWY92}), it is easy to see that the mapping $(t,\om,z)\mapsto\int_0\p tK(z,\om)_s\rmd A_s(\om)$ is jointly predictable. Now it is a standard procedure to get the claim for every $K(\cdot)$ which is real-valued and integrable with respect to $A$. The case of a complex-valued integrable $K(\cdot)$ follows immediately from this latter one and the proof of the lemma is complete.
\end{proof}
\subsection{Fourier representation of the GKW-decomposition}\label{subsec:four.met}
Let $X$ be a factor process taking values in $\Rbb\p n$. We assume that there exist an $R\in\Rbb\p n$ such that $\Ebb[\exp(2R\p\top X_T)]<\infty$ and define the `strip' $\Sscr:=\{z\in\Cbb\p n: \Rea(z)=R\}$.
A square integrable European option is given and its pay-off is $\et=h({X_T})$ for some real-valued function $h$ with domain in $\Rbb\p n$. We assume that the two-sided Laplace transform $\widetilde{h}$ of $h$ exists in $R$ and that it is integrable on $\Sscr$. Then $h$ has the following representation
\begin{equation}\label{eq:rep.fourier}
h(x)=\frac1{(2\pi\,\rmi)\p n}\,\int_{R-\rmi\,\infty}\p{R+\rmi\,\infty}\exp(\,z\p\top x)\widetilde h(z)\rmd z=\int_\Sscr \exp(\,z\p\top x)\zt(\rmd  z)\,,
\end{equation}  
where $\zt$ is the complex-valued non-atomic finite measure on $\Sscr$ defined by
\begin{equation}\label{eq:meas.zet}
\zt(\rmd  z):=\frac1{(2\pi\,\rmi)\p n}\,\widetilde h(z)\,\rmd  z\,.
\end{equation}

For each $z\in\Sscr$, the process $H(z)=(H(z)_t)_{t\in[0,T]}$ defined by $H(z)_t:=\Ebb[\exp(z\p\top X_T)|\Fscr_t]$ is a square integrable\emph{ complex-valued} martingale.  
Analogously $H=(H_t)_{t\in[0,T]}$, $H_t:=\Ebb[\et|\Fscr_t]$, is a square integrable martingale. 
We recall that we always consider \cadlag martingales.

We now make the following assumption which will be in force throughout this section.

\begin{assumption}\label{ass:joi.prog}
The mapping $(t,\om,z)\mapsto H(\om,z)_t$ is jointly progressively measurable.
\end{assumption}
Under Assumption \ref{ass:joi.prog} we can show the following result:
\begin{proposition}\label{prop:H.mart}
The estimate  $\sup_{z\in\Sscr}\Ebb[|H(z)_t|\p2]<+\infty$ holds. Furthermore, under Assumption \ref{ass:joi.prog}, the process $\widetilde H_t:=\int_\Sscr H(z)_t \zt(\rmd z)$ is a square integrable martingale which is indistinguishable from $H$ and hence the identity
\begin{equation}\label{eq:rep.Hz}
H_t=\int_\Sscr H(z)_t\zt(\rmd z)
\end{equation}
holds.
\end{proposition}
\begin{proof}
 From
\[
\sup_{z\in\Sscr}|H(z)_t|\p2\leq H(2R)_t\in L\p1(\Pbb),\quad t\in[0,T]
\]
we get the first part of the proposition. By Assumption \ref{ass:joi.prog} and  Proposition \ref{prop:poi.br.com}, (i), the process $\widetilde H$ is a square integrable martingale. To see \eqref{eq:rep.Hz}, we recall that $H$ and $\widetilde H$ are martingales (hence c\`adl\`ag) and clearly modifications of each other. Therefore they are indistinguishable and the proof of the proposition is complete.
\end{proof}

Let $S$ describe the price process of some traded asset. We assume that $S$ is a strictly positive square integrable martingale starting at $S_0>0$.  We now consider the GKW-decomposition of $H$ and $H(z)$ with respect to $S$, that is
\begin{equation}\label{eq:gkw dec.H.Hz}
H=H_0+\vartheta\cdot S+L\,,\qquad H(z)=H(z)_0+\vartheta(z)\cdot S+L(z)\,, \quad z\in\Sscr\,,
\end{equation}
where $\vartheta\in \Lrm\p2(S)$, $\vartheta(z)\in\Lrm\p2_\Cbb(S)$, $L\in\Hscr\p 2_0$ and $L(z)$ is a complex-valued square integrable martingale $z\in\Sscr$. By definition of the GKW-decomposition, $L$ and $L(z)$ are orthogonal to $S$.

\begin{theorem}\label{thm:hed.err}
Let $H$ and $H(z)$ be defined as above and let their respective GKW-decomposition be given by \eqref{eq:gkw dec.H.Hz}. Let Assumption \ref{ass:joi.prog} hold and the mapping $(t,\om,z)\mapsto\vartheta(\om,z)_t$ be jointly predictable. Then the identities
\begin{align}
&\vartheta=\int_\Sscr\vartheta(z)\zt(\rmd z)\label{eq:hed.str}\,;
\\&
\Big(\int_\Sscr\vartheta(z)\zeta(\rmd  z)\Big) \cdot S=\int_\Sscr\big(\vartheta(z)\cdot S\big)\zeta(\rmd  z)\label{eq:stoc.fub}\,;
\\&L=\int_\Sscr L(z)\zt(\rmd  z)\label{eq:hed.error}
\end{align}
hold. In particular the GKW-decomposition of $H$ is
\begin{equation}\label{eq:gkw.eu.op}
H=H_0+\Big(\int_\Sscr\vartheta(z)\zeta(\rmd  z)\Big) \cdot S+\int_\Sscr L(z)\zt(\rmd  z)\,.
\end{equation}
\end{theorem}
\begin{remark}
In a nutshell, this theorem shows that the Fourier representation \eqref{eq:rep.fourier} of a claim and its GKW-decomposition can be interchanged under very general conditions. In other words, the GKW-decomposition of the claim can be obtained by integrating the GKW-decomposition of the conditional moment generating function $H(z)_t$ in a suitable complex domain $\Sscr$ against the measure $\zt$ that determines the claim via \eqref{eq:rep.fourier}.
\end{remark}

\begin{proof} 
First we show \eqref{eq:hed.str}. Clearly $\vartheta(z)\cdot\aPP{S}{S}_T(\om)<+\infty$ a.\ s. Hence, because of
\[
\aPP{S}{H(z)}_t(\om)=\int_0\p t\vartheta(\om,z)_s\rmd \aPP{S}{S}_s(\om),
\] 
 from Lemma \ref{lem:joi.pred.int.fv}, $(t,\om,z)\mapsto\aPP{S}{H(z)}_t(\om)$ is jointly predictable. So $\int_\Sscr\aPP{S}{H(z)}\zt(\rmd z)$ is a predictable process. Proposition \ref{prop:poi.br.com} (ii) and Proposition \ref{prop:H.mart} yield
\begin{equation}\label{eq:rep.pb.SH}
\aPP{S}{H}=\int_\Sscr\aPP{S}{H(z)}\zt(\rmd z)=\int_\Sscr\int_0\p \cdot\vartheta(z)_s\rmd \aPP{S}{S}_s\zt(\rmd z).
\end{equation}
Furthermore, for every $z\in\Sscr$, the identity 
\[
\Ebb[|H(z)_t|\p2]=\Ebb[|\vartheta(z)\cdot S_t|\p 2+|L(z)_t|\p2+|H(z)_0|\p2]
\]
holds. Hence we can estimate 
\begin{equation}\label{eq:est.GKW.hz}
\Ebb[|\vartheta(z)\cdot S_t|\p 2]\leq \Ebb[|H(z)_t|\p2]\,,\qquad\Ebb[|L(z)_t|\p2]\leq \Ebb[|H(z)_t|\p2],\quad z\in\Sscr,\quad t\geq0,
\end{equation}
which, from Proposition \ref{prop:H.mart}, imply
\begin{equation}\label{eq:est.GKW.hz.sup}
\sup_{z\in\Sscr}\Ebb\left[|\vartheta(z)\cdot S_t|\p 2\right]<+\infty\,,\qquad \sup_{z\in\Sscr}\Ebb\left[|L(z)_t|\p2\right]<+\infty.
\end{equation}
 Because of 
\begin{equation}\label{eq:int.thtz.inL2}
\begin{split}
\Ebb\left[\int_0\p T\int_\Sscr|\vartheta(z)_t|\p 2|\zeta|(\rmd  z)\rmd\aPP{S}{S}_t\right]
&=\int_\Sscr\Ebb\left[\int_0\p T|\vartheta(z)_t|\p 2\rmd\aPP{S}{S}_t\right]|\zeta|(\rmd  z)
\\&
=\int_\Sscr\Ebb\left[|\vartheta(z) \cdot{S}_T|\p 2\right]|\zeta|(\rmd  z)
<\infty,
\end{split}
\end{equation}
where in the last estimation we applied \eqref{eq:est.GKW.hz.sup},  Fubini's theorem and \eqref{eq:rep.pb.SH} yield
\[
\int_0\p \cdot\int_\Sscr\vartheta(z)_s\zeta(\rmd  z)\rmd\aPP{S}{S}_s=\aPP{S}{H},
\]
which, since $\vartheta=\rmd\aPP{S}{H}/\rmd\aPP{S}{S}$, proves \eqref{eq:hed.str}. Now we show \eqref{eq:stoc.fub}. Because of \eqref{eq:hed.str}, the predictable process 
\[
\int_\Sscr\vartheta(z)\zt(\rmd z)=\int_\Sscr\vartheta(z)h(z)|\zt|(\rmd z)
\]
belongs to $L\p2(S)$, where $h$, with $|h|=1$, is the density of $\zt$ with respect to $|\zt|$. From \cite[Theorem 1 in \S 5.2]{StYo78}, there exists a jointly optional mapping $(t,\om,z)\mapsto\widetilde Y(\om,z)_t$ such that $\widetilde Y(z)$ is indistinguishable from $\vartheta(z)\cdot S$, for every $z\in\Sscr$. Hence we can apply \cite[Theorem 5.44]{J79}, to deduce that the process $Y\p\zt:=\int_\Sscr\big(\vartheta(z)\cdot S\big)\zt(\rmd z)$ is well defined and a version of
\[
\vartheta\cdot S=\int_0\p \cdot\left(\int_\Sscr\vartheta(z)_s\zeta(\rmd z)\right)\rmd S_s,
\]
and we do not distinguish these versions. This proves \eqref{eq:stoc.fub}.
In the next step we show \eqref{eq:hed.error}. From \eqref{eq:gkw dec.H.Hz}, for every $z\in\Sscr$, we get the identity
\begin{equation}\label{eq:Lz.rew}
L(z)_t= H(z)_t-H(z)_0-\vartheta(z)\cdot S_t=H(z)_t-H(z)_0-\widetilde Y(z)_t,\quad\textnormal{ a.s., }\ t\geq0.
\end{equation}
We can therefore integrate \eqref{eq:Lz.rew} with respect to $\zt$, obtaining
\[
\widetilde L_t:=\int_\Sscr \big(H(z)_t-H(z)_0-\vartheta(z)\cdot S_t\big)\zt(\rmd z)=H_t-H_0-\vartheta\cdot S_t=L_t,\quad\textnormal{ a.s., }\ t\geq0.
\]
Hence $\widetilde L$ is a version of $L$ and therefore $\Fbb$-adapted (because $\Fbb$ is complete). From \eqref{eq:est.GKW.hz} we can apply Proposition \ref{prop:poi.br.com} (i), to deduce that $\widetilde L$ is a martingale. Hence, $L$ and $\widetilde L$ are indistinguishable. The proof of the theorem is now complete. 
\end{proof}

In the proposition below, the set of parameters is $\Sscr:=\Sscr\p1\times\Sscr\p2$, where $\Sscr\p1$ and $\Sscr\p2$ are two strips of $\Cbb\p n$. Hence all joint measurability properties (see Definition \ref{def:joi.mea}) are formulated with respect to the $\sig$-algebra $\Bscr(\Sscr)=\Bscr(\Sscr\p1)\otimes\Bscr(\Sscr\p2)$.
\begin{theorem}\label{thm:pred.cov.mixt}
Let $\Sscr\p j:=\{z\in\Cbb^n: \Rea(z)=R^j\}$ with $\Ebb[\exp(2(R\p j)\p\top X_T)]<+\infty$ and let $\et\p j$ have the representation \eqref{eq:rep.fourier} on the strip $\Sscr\p j$ with respect to the measure $\zt\p j$ (cf.\ \eqref{eq:meas.zet}); let $L\p j$ denote the orthogonal component in the GKW decomposition of $H\p j=\big(\Ebb[\et\p j|\Fscr_t]\big)_{t\in[0,T]}$ with respect to $S$, for $j=1,2$. 

\textnormal{(i)} If the mapping $(t,\om,z_1,z_2)\mapsto\aPP{L(z_1)}{L(z_2)}_t(\om)$ is jointly measurable  then $(t,\om,z_1,z_2)\mapsto\Var(\aPP{L(z_1)}{L(z_2)})_t(\om)$ is jointly measurable as well and 
\begin{equation}\label{eq:joi.bb.var}\begin{split}
\Ebb\bigg[\int_{\Sscr\p1}\int_{\Sscr\p2}&\Var(\aPP{L(z_1)}{L(z_2)})_T|\zt\p2|(\rmd z_2)|\zt\p1|(\rmd z_1)\bigg]<+\infty
\end{split}
\end{equation}
holds. Moreover, if the process $D:=\int_{\Sscr\p1}\int_{\Sscr\p2}\aPP{L(z_1)}{L(z_2)}\zt\p2(\rmd z_2)\zt\p1(\rmd z_1)$ is predictable, the covariation of the square integrable martingales $L\p1$ and $L\p2$ is given by
\begin{equation}\label{eq:joi.pg}
\aPP{L\p 1}{L\p 2}=\int_{\Sscr\p1}\int_{\Sscr\p2}\aPP{L(z_1)}{L(z_2)}\zt\p2(\rmd z_2)\zt\p1(\rmd z_1)\,.
\end{equation}
and hence
\begin{equation}\label{eq:joi.pb.exp}
\Ebb[\aPP{L\p 1}{L\p 2}_T]=\int_{\Sscr\p1}\int_{\Sscr\p2}\Ebb[\aPP{L(z_1)}{L(z_2)}_T]\zt\p2(\rmd z_2)\zt\p1(\rmd z_1)\,.
\end{equation}

\textnormal{(ii)} If furthermore there exists a jointly progressively measurable complex-valued stochastic process $K(z_1,z_2)=K\p1(z_1,z_2)+\rmi K\p2(z_1,z_2)$ such that
\[
\aPP{L(z_1)}{L(z_2)}=\int_0\p \cdot K(z_1,z_2)_s\rmd s,
\]
then the identity
\begin{equation}\label{eq:for.pb.gen}
\aPP{L\p 1}{L\p 2}=\int_0\p\cdot\int_{\Sscr\p1}\int_{\Sscr\p2}K(z_1,z_2)_s\zt\p2(\rmd z_2)\zt\p1(\rmd z_1)\rmd s
\end{equation}
holds and
\begin{equation}\label{eq:for.pb.gen.exp}
\Ebb[\aPP{L\p 1}{L\p 2}_T]=\int_0\p T\int_{\Sscr\p1}\int_{\Sscr\p2}\Ebb[K(z_1,z_2)_s]\zt\p2(\rmd z_2)\zt\p1(\rmd z_1)\rmd s.
\end{equation}
\end{theorem}
\begin{proof}
From \eqref{eq:hed.error}, we have $L\p j=\int_{\Sscr\p j}L(z_j)\zt\p j(\rmd z_j)$, $j=1,2$. Furthermore, the estimation \eqref{eq:est.GKW.hz} holds with both $\Sscr=\Sscr\p1$ and $\Sscr=\Sscr\p2$. Hence, to prove this theorem one has to proceed exactly as in Proposition \ref{prop:poi.br.com} and we omit further details.
\end{proof}

We now combine in a theorem the results obtained in this section with those of Section \ref{sec:sem.quad} (for complex valued semimartingales).
\begin{theorem}\label{thm:comsec3.sec4}
Let the factor process $X$ be a locally square integrable semimartingale and $S=\rme\p {X\p1}$ be a square integrable martingale.

\textnormal{(i)} Let $f:[0,T]\times\Rbb\p n\times\Sscr\p j \longrightarrow\Cbb $ be a $\Bscr([0,T])\otimes\Bscr(\Rbb\p n)\otimes\Bscr(\Sscr)$ measurable function with $f\in C\p{1,2}([0,T]\times\Rbb\p n)$, such that $H(z)_t=f(t,X_t\p1,\ldots,X\p n,z)$. 

\textnormal{(ii)} Let $\et\p j=h\p j(S_T)$ be an European option such that $h\p j$ is a function with representation as in \eqref{eq:rep.fourier} and \eqref{eq:meas.zet} over $\Sscr\p j$,  $j=1,2$.

Then the assumptions of Proposition \ref{prop:poi.br.com}, Theorem \ref{thm:hed.err} and Theorem \ref{thm:pred.cov.mixt} (i) hold.
\end{theorem}

\begin{proof}
The joint progressively measurability of $H(\cdot)$ is clear because of the assumption (i) in the theorem. The joint predictability of $\vartheta(\cdot)$ follows from Theorem~\ref{thm:poi.br.res.smar}, in particular from \eqref{eq:var.opt.str}. Hence the assumptions of Theorem \ref{thm:hed.err} are satisfied.  We define for $z\in\Sscr\p j$ and $j=1,2$, $f\p j(\cdot,z):=f(\cdot,z_j)$ and
\begin{equation}\label{eq:proc.Wz}
\begin{split}
W\p j(t,\om,x_1,\ldots,x_n,z)&:=W (t,\om,x_1,\ldots,x_n,z_j):=\\&f(t,x_1+X\p1_{t-}(\om),\ldots,x_n+X\p n_{t-}(\om),z_j)-f(t,X\p1_{t-}(\om),\ldots,X\p n_{t-}(\om),z_j).
\end{split}
\end{equation}
The mapping $(t,\om,x_1,\ldots,x_n,z_j)\mapsto W(t,\om,x_1,\ldots,x_n,z_j)$ is $\Pscr\otimes\Bscr(\Rbb\p n)\otimes\Bscr(\Sscr\p j)$-measurable. So, from \eqref{eq:res.pb.exp}, we deduce that the mapping $(t,\om,z_1,z_2)\mapsto\aPP{L(z_1)}{L(z_2)}_t(\om)$ is jointly predictable and hence it is jointly measurable. This yields  the predictability of the process $D$ defined in Theorem \ref{thm:pred.cov.mixt}. Hence all the assumptions of  Theorem \ref{thm:pred.cov.mixt}, (i) are fulfilled. The proof of the theorem in now complete.
\end{proof}

\subsection{Variance-Optimal Hedging in Affine Stochastic Volatility Models}\label{subs:vo.smar.mod}
We now more closely discuss the case in which $(X,V):=(X\p1,X\p2)$ is an affine process and a semimartingale. An affine process $(X,V)$ in the sense of \cite{DFS03} is a stochastically continuous Markov process taking values in $\Rbb\times\Rbb_+$ such that the joint conditional characteristic function of $(X_t,V_t)$ is of the form
\begin{equation}\label{eq:char.fun.aff.pr}
\Ebb\big[\exp( u_1X_t+ u_2V_t)|\Fscr_s\big]=\exp\big(\phi_{t-s}(u_1,u_2)+\psi_{t-s}(u_1,u_2)V_s+u_1X_s\big),\quad s\leq t
\end{equation}
where $(u_1,u_2)\in\rmi\Rbb\p2$. The complex-valued functions $\phi_t$ and $\psi_t$ are the solutions of the \emph{generalized Riccati equations}
\begin{subequations}
\begin{align}
\partial_t\phi_t(u_1,u_2) &= F((\psi_t(u_1,u_2),u_2)),\quad\phi_0(u_1,u_2)=0, \label{eq:ex.Ri.eq.phi} \\
\partial_t\psi_t(u_1,u_2) &= R((\psi_t(u_1,u_2),u_2)),\quad\psi_0(u_1,u_2)=u_2,\label{eq:ex.Ri.eq.psi}
\end{align}
\end{subequations}
where $(F,R)$ are the L\'evy symbols associated with the L\'evy triplets $(\bt\p 0,\gm\p 0,\kappa\p 0)$ and $(\bt\p 1,\gm\p 1,\kappa\p 1)$ respectively. That is, setting $ u:=(u_1,u_2)\p\top$ and considering only Markov processes without killing, we have 
\begin{align*}
F(u) &:= u\p\top\bt\p0+\frac12 u\p\top\gm\p0 u+\int_{\Rbb_+\times\Rbb}\big(\rme\p{ u\p\top x}-1- u\p\top h(x_1,x_2)\big)\kappa\p0(\rmd x),\\
R(u) &:= u\p\top\bt\p1+\frac12 u\p\top\gm\p1 u+\int_{\Rbb_+\times\Rbb}\big(\rme\p{ u\p\top x}-1- u\p\top h(x_1,x_2)\big)\kappa\p1(\rmd x).
\end{align*}
Under a mild non-explosion condition, affine processes are semimartingales with absolutely continuous characteristics (cf.~\cite[Sec.~9]{DFS03}) and according to \cite[\S3]{KalMV11}, the differential characteristics $(b,c,K)$ are given by
\[
b_t= \begin{bmatrix}
       \bt_1\p0+\bt_1\p1 V_{t-}\\[.5em] 
						     \bt\p0_2+\bt_1\p2 V_{t-}
     \end{bmatrix},\quad 
		c_t=\begin{bmatrix}
		\gm\p0_{11}+\gm\p1_{11}V_{t-}&\gm\p1_{12}V_{t-}\\[.5em]
		\gm\p1_{12}V_{t-}&\gm\p1_{22}V_{t-}
		\end{bmatrix}
		,\quad
K_t(\om,\rmd x)=\kappa\p0(\rmd x)+\kappa\p1(\rmd x)V_{t-}(\om),
\]
where $\bt\p i$ belongs to $\Rbb\p2$, $\gm\p i$ is a symmetric matrix in $\Rbb\p{2\times2}$ and $\kappa\p i$ is a L\'evy measure on $\Rbb\p2$ with support in $\Rbb_+\times\Rbb$, $i=0,1$. Furthermore, $\gm\p0_{22}=\gm\p0_{12}=0$; $\bt\p0_2-\int_{\Rbb_+\times\Rbb}h\p2(x_2)\kappa_0(\rmd x)$ is well-defined and nonnegative (we recall that we define $h(x_1,x_2):=(h\p1(x_1),h\p2(x_2)):=(x_11_{\{|x_1|\leq1\}},x_21_{\{|x_2|\leq1\}})$) and we assume $\int_{\{x_2>1\}}x_2\kappa\p1(\rmd x)<+\infty$, to rule out explosion in finite time (cf. \cite[Lem.~9.2]{DFS03}). 

We now deduce some of the results of \cite{KalP07} from the theory that we have developed. Assume that $(X,V)$ is locally square integrable and that $\rme\p{zX_T}$ is square integrable for every $z$ in a given complex strip $\Sscr=\{z\in\Cbb: z=R+\rmi\Ima(z)\}$. Moreover, assume that $S=\rme\p X$ is a square integrable martingale. In this case, $F(1,0)=R(1,0)=0$, where $F$ and $R$ denote the L\'evy symbols associated with the L\'evy triplets $(\bt\p 0,\gm\p 0,\kappa\p 0)$ and $(\bt\p 1,\gm\p 1,\kappa\p 1)$ respectively. Because of the affine property, $H(z)$ takes the form
\[
H(z)_t=\rme\p{zX_t}\exp\big(\phi_{T-t}(z,0)+\psi_{T-t}(z,0)V_t\big),\quad z\in\Sscr.
\]
Hence, $f(t,x,v,z)=\rme\p{zx}\exp\big(\phi_{T-t}(z,0)+\psi_{T-t}(z,0)v\big)$, so that 
\begin{equation}\label{eq:f.aff}
\partial_{x} f(t,x,v,z)=zf(t,x,v,z),\qquad\partial_{v} f(t,x,v,z)=\psi_{T-t}(z,0)f(t,x,v,z) .
\end{equation}
The process $W$ in \eqref{eq:proc.W} is now given by $W(t,x,v,z)=H(z)_{t-}(\exp(zx+\psi_{T-t}(z,0)v)-1)$. The process $\xi$ of Corollary \ref{cor:poi.br.res.smar.abs.con} becomes 

\begin{equation}\label{eq:proc.xi}
\xi_t=\xi\p0+\xi\p1V_{t-},\qquad
\xi\p i:=\gm\p i_{11}+\int_E(\rme\p{x}-1)\p2\kappa\p i(\rmd x),\quad i=0,1.
\end{equation} 
We notice that now $\xi\p i$ are constant in time, $i=0,1$. Furthermore, setting
\begin{equation}\label{eq:def.p.aff}
p\p i:=z\gm\p i_{11}+\psi_{T-t}(z,0)\gm\p i_{12}+\int_E\big(\exp(zx+\psi_{T-t}(z,0)v)-1\big)\big(\rme\p{x}-1\big)\kappa\p i(\rmd x),\quad i=0,1,
\end{equation}
from \eqref{eq:var.opt.str.abs.con}, we deduce
\begin{equation}\label{eq:thtz.aff}
\vartheta(z)_t=\frac{H(z)_{t-}}{S_{t-}}\frac{p\p0+p\p1V_{t-}}{\xi\p0+\xi\p1 V_{t-}}
\end{equation}
which is \cite[Lemma 5.2]{KalP07}. Furthermore, from Theorems~\ref{thm:hed.err} and~\ref{thm:comsec3.sec4}, the integrand $\vartheta$ in the GKW decomposition of $H$ with respect to $S$ is given by
\[
\vartheta_t=\int_\Sscr\vartheta(z)_t\zt(\rmd z)=\int_\Sscr\frac{H(z)_{t-}}{S_{t-}}\frac{p\p0+p\p1V_{t-}}{\xi\p0+\xi\p1 V_{t-}}\zt(\rmd z)
\]
which is \cite[Theorem 4.1]{KalP07}. With a straightforward computation, from \eqref{eq:res.pb.exp.abs.con} with $f\p i_t=f(t,x,v,z_i)$, $i=1,2$, we also obtain the explicit expression of $\aPP{L(z_1)}{L(z_2)}$, which is given in \cite[Eq.~(5.10), p.\ 97]{KalP07}. The process $(t,\om,z_1,z_2)\mapsto\aPP{L(z_1)}{L(z_2)}_t(\om)$ is $\Pscr\otimes\Bscr(\Sscr)\otimes\Bscr(\Sscr)$-measurable. Therefore, from \eqref{eq:for.pb.gen} we deduce the explicit expression of $\aPP{L}{L}$.

Notice that we can obtain an exact representation of $\aPP{L}{L}$, while in \cite{KalP07}, this predictable covariation is represented only as Cauchy principal value of the right-hand side of \eqref{eq:for.pb.gen}. We also stress that we are able to compute the quantities $\vartheta(z)$, $\aPP{L(z_1)}{L(z_2)}$ and $\aPP{L}{L}$ under the only assumption that $(X,V)$ is a locally square integrable semimartingale.  In \cite[Assumption 3.1(i)]{KalP07}, stronger assumptions based on analyticity properties of $\phi$ and $\psi$ are made. According to our results, these assumptions are only needed to calculate the expectation $\Ebb[\aPP{L}{L}_T]$, but not $\aPP{L}{L}_T$ itself. 
 
\section{Applications}\label{sec:app}

In this section we apply the results of Section \ref{sec:sem.quad} and \ref{sec:GKW.eu} to two continuous stochastic volatility models: The Heston model, which is an affine model in the sense of \cite{DFS03}, and the $3/2$-model, which is a non-affine model. 

We set up a variance-optimal semi-static hedging strategy for a variance swap $\et\p0$: If $(X,V)$ is the continuous semimartingale describing the stochastic volatility model, then
\begin{equation}\label{eq:var.sw}
\et\p0=[X,X]_T-k\p\textnormal{swap}=\aPP{X\p c}{X\p c}_T-k\p\textnormal{swap}.
\end{equation}
By continuity, the process $(X,V,[X,X])$ is a locally square integrable semimartingale.
The price process $S=\rme\p X$ is assumed to be a square integrable martingale.
A basket $(\et\p1,\ldots,\et\p d)$ of European options written on $S$  is fixed and we use them to implement a variance-optimal semi-static hedging strategy for $\et\p0$. We assume that each option $\et\p j$ in the basket is square integrable and such that $\et\p j=h\p j(S_T)$, where $h\p j$ can be represented as in \eqref{eq:rep.fourier} and \eqref{eq:meas.zet} on a strip $\Sscr\p j=\{z\in\Cbb: z=R\p j+\Ima(z)\}$, with $\Ebb[\exp(2R\p jX_T)]<+\infty$; $H\p j$ is the square integrable martingale associated with $\et\p j$, that is $H\p j_t:=\Ebb[\et\p j|\Fscr_t]$, and its GKW decomposition is $H\p j=H\p j_0+\vartheta\p j\cdot S+L\p j$.

\subsection{The Heston model}\label{subsec:hes.mod}
The Heston stochastic volatility model $(X,V)$ is given by
\begin{equation}\label{eq:Heston}
\rmd X_{\,t}=-\frac12 V_{\,t}\rmd t+\sqrt{V_{\,t}} \rmd W_{\,t}\p{\,1},\quad \rmd V_{\,t}=-\lambda(V_{\,t}-\kappa )\rmd t+\sigma\sqrt{V_{\,t}} \rmd W_{\,t}\p{\,2},
\end{equation}
where $(W\p{\,1}, W\p{\,2})$ is a two-dimensional correlated Brownian motion such that $\aPP{W\p{\,1}}{W\p{\,2}}_{\,t}=\rho t$, $\rho\in[-1,1]$. Typical values of $\rho$ are negative and around $-0.7$. The parameters $\lambda$, $\sigma$ and $\kappa$ are strictly positive. This model for the dynamics of $(X,V)$ is known as \emph{Heston model}, cf. \cite{H93}. Notice that $(X,V)$ is an homogeneous Markov process with respect to $\Fbb$. Furthermore, $V_t\geq0$, $t\geq0$ (cf.\ \cite[Proposition 2.1]{AP07}). 

\paragraph{Properties of Heston model. } The continuous semimartingale $(X,V)$, whose dynamic is given by \eqref{subsec:hes.mod}, is an affine process in the sense of \cite{DFS03}. Hence for $ u=(u_1,u_2)\p\top\in\Cbb\p 2$ such that $\Ebb[\exp(u_1X_T+u_2V_T)]<+\infty$, the conditional characteristic function of $(X_T,V_T)$ is given by
\begin{equation}\label{eq:joi.char}
\Ebb[\exp( u_1X_T+u_2V_T)|\Fscr_t]=\exp(\phi_{T-t}(u)+u_1 X_t+\psi_{T-t}(u)V_t)\,,\quad t\in[0,T],
\end{equation}
where $\phi,\psi:\Cbb\p2\longrightarrow\Cbb$ and $\psi$ is the solution of the following Riccati equation:
\begin{equation}\label{eq:psi.hest}
\partial_t\psi_t(u)=\textstyle\frac12\displaystyle\sig\p{\,2}\psi_t(u)\p{\,2}-(\lm-\rho\sig u_{\,1})\psi_t(u)+\textstyle\frac12\,\displaystyle \big(u_{\,1}\p{\,2}-u_{\,1}\big),\quad\psi_0(u)=u_{\,2}\,,
\end{equation}
and 
\begin{equation}\label{eq:phi.hest}
\phi_t(u)=\lm\kappa\,\int_0\p t\psi_t(u)\,\rmd s\,.
\end{equation}
The unique solution of \eqref{eq:psi.hest} exists up to an explosion time $t_{\,+}(u_{\,1})$ \emph{which can be finite}. The analytic expression of the explosion time $t_{\,+}(u_{\,1})$ is given in \cite[Proposition 3.1]{AP07} (see also \cite{FK10} or \cite[\S~6.1]{K11}). 

Considering the root
\[\Psi=\Psi(u_{\,1}):=\frac{1}{\sig\p{\,2}}\,\Big(\lm-\rho\sig u_{\,1}-\sqrt{\Delta(u_{\,1})}\Big)\]
of the characteristic polynomial of \eqref{eq:psi.hest}, it is possible to write the explicit solution of \eqref{eq:psi.hest}. The root $\Psi$ leads to a representation of $\psi$ which is continuous in the complex plane, i.e. it does not cross the negative real axis, which is the standard branch cut for the square-root function in the complex plane, and is therefore more suitable for numerical implementations (cf.\ \cite{JK05} and \cite{AMST06}). Following \cite[Proposition 4.2.1]{Al15} -- where however the complex conjugate of $\Psi$ is used -- 
the explicit expression of $\psi_t(u)$ in its interval of existence $[0,t_+(u_1))$, is

\begin{equation}\label{eq:sol.psi}
\psi_t(u):=\begin{cases}
u_{\,2}+(\Psi-u_{\,2})\,\frac{1-\exp\big(-t\sqrt{\Delta}\big)}{1-g\exp(-t\sqrt{\Delta})},&\quad \Delta(u_{\,1})\neq0;\\\\
u_{\,2}+(\Psi-u_{\,2})\p{\,2}\,\frac{\sig\p{\,2}t}{2+\sig\p{\,2}(\Psi-u_{\,2})t},&\quad \Delta(u_{\,1})=0\,,
\end{cases}
\end{equation}
where we define
\[g(u_1,u_{\,2}):=\frac{\lm-\rho\sig u_{\,1}-\sig\p{\,2}u_{\,2}-\sqrt{\Delta(u_{\,1})}}{\lm-\rho\sig u_{\,1}-\sig\p{\,2}u_{\,2}+\sqrt{\Delta(u_{\,1})}}\]
and use the conventions
\[\frac{\exp(-t\sqrt\Delta)-g}{1-g}:=1,\qquad \frac{1-\exp(t\sqrt\Delta)}{1-g\exp(t\sqrt\Delta)}:=0\]
 whenever the denominator of $g$ is equal to zero. 
Then, from \eqref{eq:phi.hest},

\begin{equation}\label{eq:sol.phi}
\phi_t(u):=\begin{cases}
\lm\kappa(\Psi+\frac{2\sqrt{\Delta}}{\sig\p{\,2}})\,t-\frac{2\lm\kappa}{\sig\p{\,2}}\,\log\frac{\exp(t\sqrt{\Delta})-g}{1-g},&\quad \Delta(u_1)\neq0;\\\\
\lm\kappa\Psi\,t-\frac{2\lm\kappa}{\sig\p{\,2}}\,\log\Big(1+\frac{\sig\p{\,2}}{2}(\Psi-{u_{\,2}})t\Big),&\quad \Delta(u_1)=0\,.
\end{cases}
\end{equation}

We observe that, from \eqref{eq:Heston}, the quadratic covariation of $X$ is
\[
[X,X]_t:=\int_0\p tV_s\rmd\, s\,,\quad t\geq0,
\]
and the process $Z:=(X,V,[X,X])\p\top$ is affine (cf.\ \cite[Lemma 4.1]{KalMV11}). Furthermore, the moment generating function of $Z_T$ exists in an open neighbourhood of the origin in $\Rbb\p 3$, say  $B_\ep(0)$, $\ep>0$, (cf.\ \cite[Theorem 10.3(b), Lemma 10.1(c)]{F09}). Therefore, $Z_T$ possesses exponential moments in $B_\ep(0)$ and hence each component has finite moments of every order.

The dynamic of the price process $S=(S_t)_{t\in[0,T]}$, $S_t:=\exp(X_t)$, is given by
\begin{equation}\label{eq:proc.S}
\rmd S_{\,t}=S_{\,t}\sqrt{V_{\,t}}\rmd W\p{\,1}_{\,t}\,.
\end{equation}
From \cite[Corollary 3.4]{KalM10}, $S$ is a martingale.  
We assume that $S$ is square integrable, that is $S_T\in L\p2(\Pbb)$. According to \cite[Theorem 2.14, Example 2.19]{KM15}, the square integrability of $S_T\p{\, z}=\exp(z\,X_{\,T})$ for $z\in\Cbb$ is equivalent to the existence up to time $T$ of the solution of the Riccati equation \eqref{eq:psi.hest} starting at $u=(2\Rea(z),0)\p\top$, that is $T<t_+(2\Rea(z))$. In particular we assume that $T<t_+(2)$.

Let $z\in\Cbb$ be such that $z=R+\rmi\Ima(z)$, where $\Ebb[\exp(2RX_T)]<+\infty$. The complex-valued process $H(z)_t=\Ebb[\exp(zX_T)|\Fscr_t]$ is a square integrable martingale.
We set 
\begin{align}&r\p 0_{z_1,z_2}:=\phi_t(z_1,0)+\phi_t(z_2,0)\label{eq:xi0}\\\nonumber\\&r\p 1_{z_1,z_2}:=\psi_t(z_1,0)+\psi_t(z_2,0)\label{eq:xi1} \,.
\end{align}

The next formula was established in \cite[Eq.~(5.18)]{KalP07}. For completeness we give the proof in the appendix.
\begin{proposition}\label{prop:mix.mom}
Let $z_j\in\Cbb$, $z_j=R\p j+\rmi\Ima(z_j)$, with $R\p j\in \Rbb$ such that $\Ebb[\exp(2R\p jX_T)]<+\infty$, $j=1,2$. Then 
\begin{equation}\label{eq:mix.mom}
\begin{split}
\Ebb[V_tH(z_1)_tH(z_2)_t]=S_0\p{z_1+z_2}&\rme\p{r\p 0_{z_1,z_2}}\times\\&\exp\big(\phi_t(z_1+z_2,r\p 1_{z_1,z_2})+\psi_t(z_1+z_2,r\p 1_{z_1,z_2})V_0\big)\times\\&\qquad\partial_{u_2}\Big[\phi_t(z_1+z_2,u_2)+\psi_t(z_1+z_2,u_2)V_0\Big]\Big|_{u_2=r\p 1_{z_1,z_2}}.
\end{split}
\end{equation}
\end{proposition}

\paragraph{Semi-Static Variance-Optimal Hedging. }We now discuss the inner variance-optimal problem in \eqref{eq:setup} for the variance swap $\et\p 0$ in Heston model and compute the quantities $A$, $B$ and $C$ defined in \eqref{eq:hed.err}. For notation we refer to Corollary \ref{cor:poi.br.res.smar.abs.con} (ii). Setting
\[
\rmd Q_t:=\rmd\aPP{V\p c}{V\p c}_t-\frac{\rmd\aPP{X\p c}{V\p c}_t}{\rmd\aPP{X\p c}{X\p c}_t}\rmd \aPP{X\p c}{V\p c}_t=\left(c_t\p{22}-\frac{c\p{12}_t}{c\p{11}_t}c\p{12}_t\right)\rmd t,
\]
 \eqref{eq:Heston} yields
\[
c_t\p{11}=V_t,\quad c_t\p{12}=\rho\sig V_t,\quad c_t\p{22}=\sig\p2V_t,\quad c_t\p{j3}=0,\quad j=1,2,3,
\]
and hence
\begin{equation}\label{eq:def.Q.Hes}
\rmd Q_t=\sig\p2(1-\rho\p2)V_t\rmd t.
\end{equation}

We recall that for $\et\p0\in L\p2(\Pbb)$ defined as in \eqref{eq:var.sw}, we define the square integrable martingale $H\p0=(H\p0_t)_{t\in[0,T]}$ by $H\p0_t:=\Ebb[\et\p0|\Fscr_t]$, $t\in[0,T]$.
\begin{proposition}\label{prop:A.hes.mod}
Let $A:=\Ebb[\aPP{L\p 0}{L\p 0}_T]$, $L\p 0$ being the residual in the GKW decomposition of $H\p0$ with respect to $S$. Then
\[
A=\sig\p2(1-\varrho\p2)\int_0\p T\al\p2(t)[(V_0-\kappa)\rme\p{-\lm t}+\kappa]\,\rmd t,
\]
where $\al(t):=\frac{1}{\lm}\,\big(1-\rme\p{-\lm(T-t)}\big)$.
\end{proposition}
\begin{proof}
The process $(X,V,[X,X])$ is a square integrable semimartingale and the random variables $X_t$, $V_t$ and $[X,X]_t$ have finite moments of every order, for every $t\leq T$. We first show that the formula $H\p0_t=f\p0(t,V_t,[X,X]_t)$ holds, where the function $(t,v,w)\mapsto f\p0(t,v,w)$ has to be determined. From Proposition \ref{prop:mix.mom},
we get
\begin{equation}\label{eq:mom.V.aux}
\Ebb[V_t]=\Ebb[H(0)_tH(0)_tV_t]=\partial_{u_2}\big[\phi_t(0,u_2)+\psi_t(0,u_2)V_0\big]\Big|_{u_2=0}.
\end{equation}
Notice that $\Delta(0)=\lm>0$. Hence, to compute the derivative in \eqref{eq:mom.V.aux}, we take the expressions of $\phi$ and $\psi$ in \eqref{eq:sol.psi} and \eqref{eq:sol.phi} for $\Delta(u_1)\neq0$.
By direct computation we then get
\[
\partial_{u_2}\phi_t(0,u_2)\big|_{u_2=0}=\kappa(1-\rme\p{-\lm t}) \,,\quad \partial_{u_2}\psi_t(0,u_2)\big|_{u_2=0}=\rme\p{-\lm t}
\]
and therefore \eqref{eq:mom.V.aux} becomes
\begin{equation}\label{eq:mom.V}
\Ebb[V_t]=\kappa(1-\rme\p{-\lm t})+\rme\p{-\lm t}V_0.
\end{equation}
Now, using the homogeneous Markov property of $V$ with respect to the filtration $\Fbb$, we get
\[
\Ebb[V_s|\Fscr_t]=\kappa(1-\rme\p{-\lm (s-t)})+\rme\p{-\lm (s-t)}V_t,\quad s\geq t.
\]
Therefore, by Fubini's theorem,
\[
H\p0_t+k\p\textnormal{swap}=\Ebb[[X,X]_T|\Fscr_t]=\int_0\p tV_s\rmd s+\int_t\p T\big[\kappa\big(1-\rme\p{-\lm (s-t)}\big)+\rme\p{-\lm (s-t)}V_t\big]\rmd s
\]
and hence
\[
H\p0_t+k\p\textnormal{swap}=\bt(t)+\al(t)V_t+[X,X]_t,\quad t\in[0,T],
\]
where 
\[
\al(t):=\frac{1}{\lm}\,\big(1-\rme\p{-\lm(T-t)}\big),\qquad\bt(t):=\frac{\kappa}{\lm}\big(\lm(T-t)-1+\rme\p{-\lm(T-t)}\big),\quad t\in[0,T].
\]
So, we see that $H\p0_t+k\p\textnormal{swap}=f\p0(t,V_t,[X,X]_t)$ where 
\begin{equation}\label{eq:fun.f0}
f\p0(t,v,w)=\bt(t)+\al(t)v+w;\quad  \partial_{v}f\p0=\al(t).
\end{equation}
 By Corollary \ref{cor:poi.br.res.smar.abs.con} (ii), \eqref{eq:res.pb.exp.abs.con} and \eqref{eq:def.Q.Hes} we get
\[
\aPP{L\p0}{L\p0}_T=\int_0\p T\al\p2(t)\rmd Q_t=\sig\p2(1-\rho\p2)\int_0\p T\al\p2(t)V_t\rmd t.
\]
Hence
\[
A=\sig\p2(1-\rho\p2)\int_0\p T\al\p2(t)\Ebb[V_t]\rmd t.
\]
The explicit computation of $\Ebb[V_t]$ is given in \eqref{eq:mom.V} and the proof is complete.
\end{proof}

As a next step, we compute the vector $B$ in \eqref{not:var.cov}. Recall that, if $z\in\Sscr\p j$, then $\Ebb[\rme\p {zX_T}]<+\infty$. Therefore, the solution of the Riccati equation starting at $(z,0)$ exists up to time $T$, (cf.\ \cite{KM15}).

\begin{remark}\label{rem:app.main.thm} We remark that, because of the affine formula \eqref{eq:char.fun.aff.pr} and by continuity of (X,V) $(t,\om,z)\mapsto H(\om,z)_t$ is $\Pscr\otimes\Bscr(\Sscr\p j)$-measurable, for every $j=1,\ldots,d$. Therefore, from Theorems~\ref{thm:hed.err} and~\ref{thm:comsec3.sec4} we get
\[
\vartheta\p j=\int_{\Sscr\p j}\frac{H(z)_t}{S_t}\big(z+\sig\rho\psi_{T-t}(z,0)\big)\zt\p j(\rmd z).
\]
\end{remark}
Notice that in the following two propositions we can consider the determination of $\psi$ and $\phi$ in \eqref{eq:sol.psi} and \eqref{eq:sol.phi} for $\Delta(u_1)\neq0$. Indeed, there exist only two complex numbers $u_1\p j$, $j=1,2$, such that $\Delta(u_1\p j)=0$ and, by assumption, $\zt$ is a non-atomic measure.

\begin{proposition}\label{prop:comp.B}
Let $B\in\Rbb\p d$ be defined as in \eqref{not:var.cov}. Then, for $j=1,\ldots,d$,
\[
B\p j=\sig\p2(1-\varrho\p2)\int_0\p T\int_{\Sscr\p j}\al(t)\Ebb\big[ H(z)_tV_t\big]\psi_{T-t}(z,0)\zt\p j(\rmd z)\rmd t
\]
where
\[
\Ebb[H(z)_tV_t]=S_0\p{z}\rme\p{r\p 0_{z,0}}\exp\big(\phi_t(z,r\p 1_{z,0})+\psi_t(z,r\p 1_{z,0})V_0\big)\partial_{u_2}\Big[\phi_t(z,u_2)+\psi_t(z,u_2)V_0\Big]\Big|_{u_2=r\p 1_{z,0}}
\]
\end{proposition}
\begin{proof} The components of the vector $B=(B\p1,\ldots,B\p d)\p\top$ are $B\p j=\Ebb[\aPP{L\p 0}{L\p j}_T]$, $j=1,\ldots,d$, where $L\p j$ is the residual in the GKW decomposition of $H\p j$ with respect to $S$, $j=1,\ldots,d$. We start computing $\aPP{L\p0}{L(z)}$, where $L\p0$ and $L(z)$ are the residuals in the GKW-decomposition of $H\p0$ and $H(z)$ (see \eqref{eq:gkw dec.H.Hz}) with respect to $S$, respectively. From \eqref{eq:fun.f0} and \eqref{eq:f.aff}, according to \eqref{eq:res.pb.exp.abs.con} (applied now to the complex-valued square integrable martingale $L(z)$) and \eqref{eq:def.Q.Hes}, we deduce
\[
\begin{split}
\rmd\aPP{L\p0}{L(z)}_t&=\partial_{v}f(t,X_t,V_t,z)\partial_{v}f\p0(t,X_t,V_t,[X,X]_t)\rmd Q_t\\&=\sig\p2(1-\rho\p2)\al(t)\psi_{T-t}(z,0)H(z)_tV_t\rmd t.\end{split}
\]
Using the affine formula \eqref{eq:joi.char}, the process $K(z)_t:=\sig\p2(1-\rho\p2)\al(t)\psi_{T-t}(z,0)H(z)_tV_t$ is jointly progressively measurable. Therefore, from Theorem \ref{thm:comsec3.sec4}  we can apply Theorem \ref{thm:hed.err} and deduce
\[
L\p j=\int_{\Sscr\p j}L(z)\zt\p j(\rmd z).
\]
Hence, because $\sup_{z\in\Sscr}\Ebb[|L(z)_T|\p2]\leq\sup_{z\in\Sscr}\Ebb[|H(z)_T|\p2]<+\infty$, Proposition \ref{prop:poi.br.com} (iii) yields
\begin{equation}\label{eq:pb.L0.Lj.hest}
\aPP{L\p0}{L\p j}_T=\sig\p2(1-\rho\p2)\int_0\p T\int_{\Sscr\p j}\al(t)\psi_{T-t}(z,0)H(z)_tV_t\zt\p j(\rmd z)\rmd t\,,\quad j=1,\ldots,d.
\end{equation}
To compute $B\p j$, we apply Theorem \ref{thm:comsec3.sec4}, Proposition \ref{prop:poi.br.com} (ii) and Fubini's theorem to exchange expectation and integrals on the right-hand side of \eqref{eq:pb.L0.Lj.hest}. The explicit expression of $\Ebb[H(z)_tV_t]$ is now given by Proposition \ref{prop:mix.mom} setting $z_1=z$ and $z_2=0$. The proof is now complete
\end{proof}

As a last step, we compute the covariance matrix $C$.
\begin{proposition}\label{prop:cov.mat.C}
Let $C=(C\p{ij})_{i,j=1,\ldots,d}$ be defined as in \eqref{not:var.cov}. Then, for $i,j=1,\ldots,n$,
\[\begin{split}
C\p{ij}=\sig\p2(1-\varrho\p2)\int_0\p T\int_{\Sscr\p i}\int_{\Sscr\p j}\psi_{T-t}(z_1,0)\psi_{T-t}(z_2,0)\Ebb\big[&H(z_1)_tH(z_2)_tV_t\big]\zt\p j(\rmd z_2)\zt\p i(\rmd z_1)\rmd t\,,
\end{split}\]
where the explicit expression of $\Ebb\big[H(z_1)_tH(z_2)_tV_t\big]$ is given by Proposition \ref{prop:mix.mom}.
\end{proposition}
\begin{proof}
By definition, we have $C\p {ij}=\Ebb\big[\aPP{L\p i}{L\p j}_T]$, $j=1,\ldots, n$.
Furthermore, from \eqref{eq:res.pb.exp.abs.con}, we get
\[
\begin{split}
\rmd\aPP{L(z_1)}{L(z_2)}_t&=\partial_{v}f(t,X_t,V_t,z_1)\partial_{v}f(t,X_t,V_t,z_2)\rmd Q_t
\\&=\sig\p2(1-\varrho\p2)\psi_{T-t}(z_1,0)\psi_{T-t}(z_2,0)H(z_1)_tH(z_2)_tV_t\rmd t
\\&=:K(z_1,z_2)\rmd t,
\end{split}
\]
where $f$ is given in \eqref{eq:f.aff}. By the affine formula, $K(z_1,z_2)$ is a jointly predictable process. From Theorem \ref{thm:comsec3.sec4} and Theorem \ref{thm:hed.err} we deduce
\[
L\p i=\int_{\Sscr\p i}L(z)\zt\p i(\rmd z),\qquad L\p j=\int_{\Sscr\p j}L(z)\zt\p j(\rmd z),\quad i,j=1,\ldots,d.
\]
Theorem \ref{thm:pred.cov.mixt} (ii) now yields
\[
\aPP{L\p i}{L\p j}_T=\sig\p2(1-\varrho\p2)\int_0\p T\int_{\Sscr\p i}\int_{\Sscr\p j}\psi_{T-t}(z_1,0)\psi_{T-t}(z_2,0)H(z_1)_tH(z_2)_tV_t\zt\p j(\rmd z_2)\zt\p i(\rmd z_1)\rmd t.
\]
From Proposition \ref{prop:poi.br.com} (ii) and Proposition \ref{prop:mix.mom}, the claim of the proposition follows and the proof is complete.

\end{proof}

\subsection{The 3/2-Model}\label{ssec:3/2.mod}
We consider the bivariate continuous stochastic volatility model described by the continuous semimartingale $(X,V)$, where
\begin{equation}\label{eq:3/2.mod}
\rmd X_t=-\frac{V_t}{2}\rmd t+\sqrt{V_t}\rmd W_t\p1,\qquad \rmd V_t=V_t(\lm-\kappa V_t)\rmd t+\sig V_t\p{3/2}\rmd W\p  2_t\,,\quad V_0>0.
\end{equation}
This model is usually called the 3/2-model and has been considered e.\ g.\ in \cite{lewis2000option, CS07}. As in the Heston model,  $(W\p{\,1}, W\p{\,2})$ is a correlated two-dimensional Brownian motion with predictable covariation $\aPP{W\p{\,1}}{W\p{\,2}}_{\,t}=\rho t$, $\rho\in[-1,1]$. To have a well-defined model, we assume $\kappa\geq-\sig\p2/2$ (see \cite[Eq.~(3)]{D12}) which ensures non-explosion of $V$ in finite time. Notice that the non-explosion condition is always satisfied whenever $\kappa>0$, as we assume. Under the non-explosion condition $V_t>0$ for all $t\in[0,T]$. Note that the two dimensional semimartingale $(X,V)$ given by \eqref{eq:3/2.mod} is \emph{not an affine process}.

For the computation of the conditional moment generating function we mainly refer to \cite{G16}, that is we regard $3/2$-model as a special case of the so-called $4/2$-model. More precisely, setting $a=0$, $b=1$ and $r=0$ in \cite[Eq.~(2.1)]{G16}, we get the dynamics of the price process $S=\rme\p X$ in the $3/2$-model under the local martingale measure,  and hence of $X=\log(S)$. We stress that, denoting by $\kappa_G, \theta_G, \sig_G$ the parameters in \cite[Eq.~(2.2)]{G16}, the relation between $\lm$, $\kappa$ and $\sig$ in \eqref{eq:3/2.mod} and $\kappa_G, \theta_G, \sig_G$ is
\begin{equation}\label{eq:par.gra.rel}
\kappa=\kappa_G\theta_G-\sig\p2_G,\quad \lm=\kappa_G,\quad \sig=-\sig_G.
\end{equation}

In \cite[Proposition 6.3.2.1]{JCY09} the transition density of the process process $R=(R_t)_{t\in[0,T]}$, $R_t:=1/V_t$ is given. Therefore, denoting by $f_{R_t}(\cdot)$ and $f_{V_t}(\cdot)$ the density function of the distribution of $R_t$ and $V_t$, respectively, and using the relation $f_{V_t}(v)=x\p{-2}f_{R_t}(v\p{-1})$, we obtain the density function of $V_t$, which is given by $f_{V_t}(v)=0$ for $v\leq0$ and for $v>0$
\begin{equation}\label{eq:den.fun.V_t}
\begin{split}
f_{V_t}(v)=\frac{1}{2v\p2}\,\frac{4\lm}{\sig\p2(\rme\p{\lm t}-1)}\,\exp\bigg(\lm\Big(1+\frac{ q}2\Big)t-\frac{2\lm}{\sig\p2(\rme\p{\lm t}-1)}\Big(\frac{1}{V_0}+\frac{\rme\p{\lm t}}{v}&\Big)\bigg)\bigg(\frac{V_0}v\bigg)\p{q/2}\times\\&I_q\bigg(\frac{4\lm\rme\p{\frac{\lm}{2}t}}{\sig\p2(\rme\p{\lm t}-1)\sqrt{V_0v}}\bigg),
\end{split}
\end{equation} 
where $q:=2(\ka+\sig\p2)/\sig\p2-1$ and $I_p(\cdot)$ denotes the modified Bessel function of the first kind. Knowing the density of $V_t$, we can compute the (non-integer) moments of $V_t$. 

From \cite[p17]{G16}, denoting by $\M(a,b,\cdot)$ the confluent hypergeometric function, we get
\begin{equation}\label{eq:mom.v_t}
\begin{split}
\Ebb\big[V_t\p{\et}\big]=\left(\frac{2\lm\rme\p{\lm t}}{\sig\p2(\rme\p{\lm t}-1)}\right)\p{\et}\exp\left(-\frac{2\lm}{\sig\p2(\rme\p{\lm t}-1)V_0}\right)&\frac{\Gamma(q+1-\et)}{\Gamma(q+1)}\times\\&\M\left(q+1-\et,q+1,\frac{2\lm}{\sig\p2(\rme\p{\lm t}-1)V_0}\right)\,,
\end{split}
\end{equation}
for $\et\in\Rbb$ such that $q+1-\et>0$. By definition of $q$ and because $\kappa>0$, this condition is in particular satisfied if $\et=2$. Furthermore, there exists an $\ep>0$ such that \eqref{eq:mom.v_t} is true also for $2+\ep$: In $3/2$-model, contrarily to Heston model, $V_t$ does not have finite moments of every order. However, \eqref{eq:mom.v_t} is a useful formula to infer, in function of the parameters $\kappa$ and $\sig\p2$ of the model, up to which number $\et>2$, the random variable $V_t\p \et$ is integrable: If, for example, $\kappa=1$ and $\sig\p2=0.2$, then $\Ebb[V_t\p {\et}]<+\infty$ for $\et<13$.

In \cite[Proposition 3.1]{G16}, conditions on the parameters of the model are given to ensure integrability of $\exp(zX_T)$, which read as $\Rea(z)\in\Ascr_{0,+\infty}$, where the set $\Ascr_{0,+\infty}$ is defined in \cite[Eq.~(3.7)]{G16}.

For $z\in\Cbb$ such that $\exp(zX_T)$ is integrable, in the next step we deduce the explicit expression of the square integrable complex-valued martingale $H(z)_t:=\Ebb[\exp(zX_T)|\Fscr_t]$, $t\geq0$. 

Recalling the relation $\M(x,y,z)=\rme\p z\,\M(y-x,y,-z)$, $z\in\Cbb$, for the confluent hypergeometric function, from  \cite[Eq.~(3.3)]{G16} with $a=0$ and $b=1$,
we obtain
\begin{equation}\label{eq:cond.char.fun}
H(z)_t=\rme\p{zX_t}g(t,V_t,z)\,,
\end{equation}
where the function $g$ is given by
\begin{equation}\label{eq:cond.char.fun.h}
g(t,v,z):=\frac{\Gamma(\bt_z-\al_z)}{\Gamma(\bt_z)}\left(\frac{2\lm}{v\sig\p2(\rme\p{\lm(T-t)}-1)}\right)\p{\al_z}\M\left(\al_z,\bt_z,-\frac{2\lm}{v\sig\p2(\rme\p{\lm(T-t)}-1)}\right)
\end{equation}
and
\begin{equation}\label{eq:def.alz.btz}
\al_z=-\frac12-\frac{\tilde\kappa_z}{\sig\p2}+c_z\,,\quad \bt_z=1+2c_z\,,\quad\tilde\kappa_z=k-z\varrho\sig\,,\quad c_z:=\sqrt{\left(\frac12+\frac{\tilde\kappa_z}{\sig\p2}\right)\p2+\frac{z-z\p2}{\sig\p2}}\,.
\end{equation}
Taking $m_z$ as in \cite[Eq.~(3.5)]{G16} with $a=0$ and $b=1$, then $c_z=\frac{m_z}2$. 
Using the properties of the confluent hypergeometric function $\M(x,y,z)$, we get the derivative with respect to $v$ of $g$ in \eqref{eq:cond.char.fun.h}:
\begin{equation}\label{eq:g.der.v}
\begin{split}
\partial_{v} &g(t,v,z)=\\&\ \ \ \frac{\al_z}{v}\big(\gm(t,v)\big)\p{\al_z}\frac{\Gamma(\bt_z-\al_z)}{\Gamma(\bt_z)}\bigg[\frac{\gm(t,v)}{\bt_z}\M\big(\al_z+1,\bt_z+1,-\gm(t,v)\big)-\M\big(\al_z,\bt_z,-\gm(t,v)\big)\bigg]\,
\end{split}
\end{equation}
where
\begin{equation}\label{eq:def.gm}
\gm(t,v):=\frac{2\lm}{\sig\p2(\rme\p{\lm(T-t)}-1)v}.
\end{equation}

To ensure that the price process $S=\rme\p X$ is a true martingale the so-called \emph{Feller condition}, which reads 
\begin{equation}\label{eq:mar.con.S}
\kappa-\varrho\sig\geq-\frac{\sig\p2}{2},
\end{equation}
is sufficient (see \cite[Eq.~(4)]{D12} or \cite[Remark 3.3]{G16} with the identification of parameter in \eqref{eq:par.gra.rel}). Notice that, if $\varrho\leq0$, then \eqref{eq:mar.con.S} is always satisfied. To get square integrability of $S_t$, $t\in[0,T]$, we require $2\in\Ascr_{0,+\infty}$.

In the following, we consider square integrable contingent claims $\et\p j, j=1,\dotsc,d$ and denote by $H\p j$ the associated square integrable martingales$H\p j_t:=\Ebb[\et\p j|\Fscr_t]$. We assume that the pay-off function $h\p j$ of $\et\p j$ has the representation in \eqref{eq:rep.fourier} with respect to a complex-valued non-atomic finite measure $\zt\p j$ on the strip $\Sscr\p j=\{z\in\Cbb: z=R\p j+\Ima(z)\}$, where $\Ebb[\exp(2R\p jX_T)]<+\infty$. 
In addition, we consider the variance swap $\et\p0=[X,X]_T-k\p{\textnormal{swap}}$. Analyzing the explicit expression of the Laplace transform  of $\et\p0$, which was calculated in \cite[Theorem 3]{CS07} (cf.\ also \cite[Appendix A]{G16}), it follows that it is defined in an open neighbourhood of zero. Therefore $\et\p0$ has finite moments of every order and we can consider the associated martingale $H\p 0_t:=\Ebb[\et\p0|\Fscr_t]$, $t\in[0,T]$, which is, in particular, square integrable.

\paragraph*{Variance-Optimal Semi-Static Hedging. } As for the Heston model we define
\[
\rmd Q_t:=\rmd\aPP{V\p c}{V\p c}_t-\frac{\rmd\aPP{X\p c}{V\p c}_t}{\rmd\aPP{X\p c}{X\p c}_t}\rmd \aPP{X\p c}{V\p c}_t=\left(c_t\p{22}-\frac{c\p{12}_t}{c\p{11}_t}c\p{12}_t\right)\rmd t,
\]
from \eqref{eq:3/2.mod} we see that 
\[
c_t\p{11}=V_t,\quad c_t\p{12}=\rho\sig V_t\p2,\quad c_t\p{22}=\sig\p2V_t\p3,\quad c_t\p{j3}=0,\quad j=1,2,3,
\]
and hence
\begin{equation}\label{eq:def.Q.th.ex}
\rmd Q_t=\sig\p2(1-\rho\p2)V_t\p3\rmd t.
\end{equation}
\begin{proposition}\label{prop:A.3/2}
Let $A=\Ebb[\aPP{L\p0}{L\p0}_T]$, $L\p0$ being the residual in the GKW-decomposition of $H\p0$ with respect to $S$. Then
\[
\begin{split}
A=&\frac{\sig\p2(1-\rho\p2)}{2}V_0\int_0\p T\int_0\p\infty\Big\{\frac{4}{\sig\p2\lm(\rme\p{\lm t}-1)}\exp\Big(\lm\Big(1+\frac{ q}{2}\Big)t-\frac{2\lm}{\sig\p2(\rme\p{\lm t}-1)}\Big(\frac1{V_0}+\frac{\rme\p{\lm t}}{v}\Big)\Big)\Big(\frac{v}{V_0}\Big)\p{1-q/2}\times
\\&
\hspace{4cm} \Big(\big(\rme\p{\lm(T-t)}-1\big)h\p\prime\Big(\frac{\rme\p{\lm(T-t)}-1}{\lm}v\Big)\Big)\p2I_q\Big(\frac{4\lm\rme\p{\frac{\lm}{2}t}}{\sig\p2(\rme\p{\lm t}-1)\sqrt{vV_0}}\Big)\Big\}\rmd v\,\rmd t\,,
\end{split}
\]
where $h$ is given by
\begin{equation}\label{eq:fct.h.swap}
h(y):=\int_0\p y\rme\p{-{2/x\sig\p2}}x\p{{2k/\sig\p2}}\int_x\p\infty \frac{2}{\sig\p2}\rme\p{{2/u\sig\p2}}u\p{-{2k/\sig\p2}-2}\rmd u\,\rmd x\,
\end{equation}
with derivative
\begin{equation}\label{eq:fct.h.swap.der}
h\p\prime(y)=\rme\p{-{2/y\sig\p2}}y\p{{2k/\sig\p2}}\int_y\p\infty \frac{2}{\sig\p2}\rme\p{{2/u\sig\p2}}u\p{-{2k/\sig\p2}-2}\rmd u\,.
\end{equation}
\end{proposition}
\begin{proof}
By continuity, the semimartingale $(X,V,[X,X])$ is locally square integrable. To compute $A$ we start from \eqref{not:var.cov}. Because of the Markov property of $V$ with respect to $\Fbb$ and \cite[Theorem 4]{CS07}, we have 
\begin{equation}\label{eq:H0.3/2}
\begin{split}
H\p0_t+k\p{\textnormal{swap}}&=\int_0\p t V_s\rmd s+\Ebb\left[\int_t\p TV_s\rmd s\Big|\Fscr_t\right]
\\&=
\int_0\p t V_s\rmd s+h\Big(\frac{\rme\p{\lm(T-t)}-1}{\lm}V_t\Big).
\end{split}
\end{equation}
The function $h$ in \eqref{eq:fct.h.swap} is twice continuously differentiable and satisfies an ODE of the second order (cf.\ \cite[Eq.~(81)]{CS07}).
Hence we see that $H\p0_t=f\p0(t,V_t,[X,X]_t)$, where
\begin{equation}\label{eq:fun.form.H0.th}
f\p0(t,v,w)=h\Big(\frac{\rme\p{\lm(T-t)}-1}{\lm}v\Big)+w;\quad\partial_{v}f\p0=h\p\prime\Big(\frac{\rme\p{\lm(T-t)}-1}{\lm}v\Big)\frac{\rme\p{\lm(T-t)}-1}{\lm}.
\end{equation}
 By Corollary \ref{cor:poi.br.res.smar.abs.con} (ii), \eqref{eq:res.pb.exp.abs.con} and \eqref{eq:def.Q.th.ex} we obtain
\[
\aPP{L\p0}{L\p0}_T=\sig\p2(1-\rho\p2)\int_0\p T\Big\{h\p\prime\Big(\frac{\rme\p{\lm(T-t)}-1}{\lm}V_t\Big)\frac{\rme\p{\lm(T-t)}-1}{\lm}\Big\}\p2V_t\p3\rmd t.
\]
Therefore,
\[
\begin{split}A&=\Ebb[\aPP{L\p0}{L\p0}_T]\\&=\sig\p2(1-\rho\p2)\int_0\p T\Big(\frac{\rme\p{\lm(T-t)}-1}{\lm}\Big)\p2\Ebb\Big[h\p\prime\Big(\frac{\rme\p{\lm(T-t)}-1}{\lm}V_t\Big)\p2V_t\p3\Big]\rmd t\,.
\end{split}
\]
To complete the proof, it is now sufficient to compute the expectation in the previous formula using the density of $V_t$ and \eqref{eq:fct.h.swap.der}. Notice that, because of Fubini's theorem, the inner expectation is finite for almost all $t\in[0,T]$. The proof of the proposition is now complete.
\end{proof}

\begin{remark}\label{rem:app.main.thm.3/2} We remark that, from \eqref{eq:cond.char.fun} and \eqref{eq:cond.char.fun.h}, the mapping $(t,\om,z)\mapsto H(\om,z)_t$ is $\Pscr\otimes\Bscr(\Sscr\p j)$-measurable, for every $j=1,\ldots,d$, by continuity of $(X,V)$. Therefore, we can apply Theorems~\ref{thm:hed.err} and~\ref{thm:comsec3.sec4} to obtain
\[
\vartheta\p j=\int_{\Sscr\p j}\frac{1}{S_t}\big(zH(z)_t+\sig\rho\partial_v g(t,V_t,z)V_t\big)\zt\p j(\rmd z).
\]
\end{remark}
As a next step, we compute the vector $B$ of the covariation of $\et\p0$ with $\et\p1,\ldots,\et\p d$, using formula \eqref{not:var.cov}. We recall that $\et\p1,\ldots,\et\p d$ are square integrable European options with representation as in \eqref{eq:rep.fourier}, while $H\p j$ denotes the square integrable martingale associated with $\et\p j$, $j=1,\ldots,d$.
\begin{proposition}\label{prop:B.3/2}
Let $B=(B\p1,\ldots,B\p d)\p\top$ be defined as in \eqref{not:var.cov}. If, for $j=1,\ldots,d$, the square integrable European option $\et\p j$ has the representation \eqref{eq:rep.fourier}, then
\[
\begin{split}
B\p j&=\frac{\sig\p2(1-\rho\p2)}{2}V_0\int_0\p T\int_{\Sscr\p j} \int_0\p\infty\Big\{\exp\Big(zX_0+\lm\Big(\frac{q}{2}+1-\frac{z\rho}{\sig}\Big)t-\frac{2\lm}{\sig\p2(\rme\p{\lm t}-1)}\Big(\frac1{V_0}+\frac{\rme\p{\lm t}}{v}\Big)\Big)\times
 \\&\hspace{4cm}\times \frac{4}{\sig\p2(\rme\p{\lm t}-1)}h\p\prime\Big(\frac{\rme\p{\lm(T-t)}-1}{\lm}v\Big)\big(\rme\p{\lm(T-t)}-1\big)\partial_{v}g(t,v,z)\times \\&\hspace{5cm}\times I_{\frac2{\sig\p2}\sqrt{B_z}}\Big(\frac{4\lm\rme\p{\frac{\lm}{2}t}}{\sig\p2(\rme\p{\lm t}-1)\sqrt{vV_0}}\Big)\Big(\frac{v}{V_0}\Big)\p{z\rho/\sig+1-q/2}
 \Big\}\rmd v\,\zt\p j(\rmd z)\rmd t,
\end{split}
\]
where
\begin{equation}\label{eq:def.Bz}
B_z:=\left(\kappa+\frac{\sig\p2}{2}\right)\p2+2\sig\p2\left[z\left(\frac{\rho}{\sig}\left(\kappa+\frac{\sig\p2}{2}\right)-\frac12\right)+\frac{z\p2}{2}(1-\rho\p2)\right]\,.
\end{equation}

\end{proposition}
\begin{proof} 
The proof is completely analogous to the one of Proposition \ref{prop:comp.B}, at least up to the computation of the expectation of $\aPP{L\p0}{L\p j}_T$.
The components of the vector $B=(B\p1,\ldots,B\p d)\p\top$ are $B\p j=\Ebb[\aPP{L\p 0}{L\p j}_T]$, $j=1,\ldots,d$, where $L\p j$ is the residual in the GKW decomposition of $H\p j$ with respect to $S$, $j=1,\ldots,d$. We start computing $\aPP{L\p0}{L(z)}$, where $L\p0$ and $L(z)$ are the residuals in the GKW-decomposition of $H\p0$ and $H(z)$ (see \eqref{eq:gkw dec.H.Hz}) with respect to $S$, respectively. From \eqref{eq:fun.form.H0.th}, \eqref{eq:cond.char.fun}, \eqref{eq:cond.char.fun.h} and \eqref{eq:def.Q.th.ex}, because of \eqref{eq:res.pb.exp.abs.con}, we deduce:
\[
\rmd\aPP{L\p 0}{L(z)}_t=\sig\p2(1-\rho\p2)h\p\prime\Big(\frac{\rme\p{\lm(T-t)}-1}{\lm}V_t\Big)\frac{\rme\p{\lm(T-t)}-1}{\lm}\partial_vg(t,V_t,z)V_t \p3\rme\p{zX_t}\rmd t=:K(z)_t\rmd t.
\]
Clearly, $(t,\om,z)\mapsto K(\om,z)_t$ is a jointly predictable mapping. Furthermore, from Theorem \ref{thm:comsec3.sec4} we can apply Theorem \ref{thm:hed.err} to get
\[
L\p j=\int_{\Sscr\p j}L(z)\zt(\rmd x),\quad j=1,\ldots,d.
\]
Because of $\sup_{z\in\Sscr\p j}\Ebb[|L(z)_t|\p2]<+\infty$, we can apply Proposition \ref{prop:poi.br.com} (iii) and deduce
\[
\begin{split}
\aPP{L\p0}{L\p j}_t&=\int_{\Sscr\p j}\aPP{L\p 0}{L(z)}_t\zt\p j(\rmd z)
\\&=\sig\p2(1-\rho\p2)\int_0\p t\int_{\Sscr\p j}h\p\prime\Big(\frac{\rme\p{\lm(T-s)}-1}{\lm}V_s\Big)\frac{\rme\p{\lm(T-s)}-1}{\lm}\rme\p{zX_s}\partial_vg(s,V_s,z)V_s\p3\zt\p j(\rmd z)\rmd s.
\end{split}
\]
Therefore, Proposition \ref{prop:poi.br.com} (ii) and Fubini's theorem yield
\begin{equation}\label{eq:Bj.3/2.mod}
B\p j=\sig\p2(1-\rho\p2)\int_0\p T\int_{\Sscr\p j}\Ebb\Big[h\p\prime\Big(\frac{\rme\p{\lm(T-s)}-1}{\lm}V_s\Big)\frac{\rme\p{\lm(T-s)}-1}{\lm}\partial_vg(s,V_s,z)V_s\p3\rme\p{zX_s}\Big]\zt\p j(\rmd z)\rmd s,
\end{equation}
for $j=1,\ldots,d$. We now compute the inner expectation in \eqref{eq:Bj.3/2.mod}. We have
\begin{equation}\label{eq:exp.first}
\begin{split}
\Ebb\Big[h\p\prime\Big(\frac{\rme\p{\lm(T-t)}-1}{\lm}V_t\Big)&\frac{(\rme\p{\lm(T-t)}-1)}{\lm}V_t\p3\partial_{v}g(t,V_t,z)\rme\p{zX_t}\Big]
\\&=
\Ebb\Big[h\p\prime\Big(\frac{\rme\p{\lm(T-t)}-1}{\lm}V_t\Big)\frac{(\rme\p{\lm(T-t)}-1)}{\lm}V_t\p3\partial_{v}g(t,V_t,z)\Ebb\big[\rme\p{zX_t}|V_t\big]\Big]\,.
\end{split}
\end{equation}
To compute the conditional expectation $\Ebb[\rme\p{zX_t}|V_t]$ we apply the results of \cite{G16}. If $(R_t)_{t\in[0,T]}$ denotes the volatility process of \cite{G16} (cf.\ \cite[Eq.~(2.2)]{G16}), we have $R_t=V_t\p{-1}$. Because $V_t$ does not vanish nor explode, the $\sig$-algebras generated by $V_t$ and $R_t$ coincide. Therefore, from \cite[Proposition 4.1]{G16}, with $a=0$ and $b=1$, recalling \eqref{eq:par.gra.rel}, we get
\begin{equation}\label{eq:con.cha.X.v}
\begin{split}
\Ebb\big[\rme\p{zX_t}|V_t\big]=\exp\left(zX_0-z\frac{\lm\rho}{\sig}\,t\right)\left(\frac{V_t}{V_0}\right)\p{\frac{z\rho}{\sig}}\frac{I_{\frac{2}{\sig\p2}\,\sqrt{B_z}}\left(\frac{4\lm\rme\p{\frac{\lm}{2}t}}{\sig\p2(\rme\p{\lm t}-1)\sqrt{V_0V_t}}\right)}{I_q\left(\frac{4\lm\rme\p{\frac{\lm}{2}t}}{\sig\p2(\rme\p{\lm t}-1)\sqrt{V_0V_t}}\right)},
\end{split}
\end{equation}
where $I_p(x)$ denotes the modified Bessel function of the first kind and $B_z$ is as in \eqref{eq:def.Bz}. Hence, inserting \eqref{eq:con.cha.X.v} in \eqref{eq:exp.first} and using the expression of the density of $V_t$ (cf.\ \eqref{eq:den.fun.V_t}), the statement follows from \eqref{eq:Bj.3/2.mod}.
\end{proof}
We now compute the covariance matrix $C$.
\begin{proposition}\label{prop:com.C.3/2.mod}
Let $C=(C\p{ij})_{i,j=1,\ldots,d}$ be defined as in \eqref{not:var.cov}. Then, for $i,j=1,\ldots,d$,
\[
\begin{split}
C\p{ij}&=	
\frac{\sig\p2(1-\rho\p2)}{2}V_0\int_0\p T\int_{\Sscr\p i}\int_{\Sscr\p j}\int_0\p\infty\Big\{\exp
\Big((z_1+z_2)X_0+\lm\Big(\frac{q}{2}+1-\frac{z_1+z_2}{\sig}\rho\Big)t\Big)\times
\\&
\hspace{2cm}\frac{4\lm}{\sig\p2(\rme\p{\lm t}-1)}\exp\Big(-\frac{2\lm}{\sig\p2(\rme\p{\lm t}-1)}\Big(\frac1{V_0}+\frac{\rme\p{\lm t}}{v}\Big)\Big)\partial_{v}g(t,v,z_1)\partial_{v}g(t,v,z_2)\times\\&\hspace{2.5cm}\times\Big(\frac{v}{V_0}\Big)\p{(z_1+z_2)\rho/\sig+1-q/2} I_{\frac{2}{\sig\p2}\sqrt{B_{z_1+z_2}}}\Big(\frac{4\lm\rme\p{\frac{\lm}{2}t}}{\sig\p2(\rme\p{\lm t}-1)\sqrt{vV_0}}\Big)\Big\}\rmd v\,\zt\p j(\rmd z_2)\,\zt\p i(\rmd z_1)\,\rmd t,
\end{split}
\]
where $g$ is defined in \eqref{eq:cond.char.fun.h} and $B_z$ in \eqref{eq:def.Bz}. The expression of the partial derivative $\partial_v g$ is given in \eqref{eq:g.der.v}.
\end{proposition}
\begin{proof}
We proceed as in Proposition \ref{prop:cov.mat.C}: By definition, we have $C\p {ij}=\Ebb\big[\aPP{L\p i}{L\p j}_T]$, $j=1,\ldots, d$. Furthermore, from \eqref{eq:res.pb.exp.abs.con} and \eqref{eq:def.Q.th.ex}, we get
\[
\begin{split}
\rmd\aPP{L(z_1)}{L(z_2)}_t&=\partial_{v}f(t,X_t,V_t,z_1)\partial_{v}f(t,X_t,V_t,z_2)\rmd Q_t
\\&=\sig\p2(1-\varrho\p2)\partial_vg(t,V_t,z_1)\partial_vg(t,V_t,z_2)V_s\p3\rme\p{(z_1+z_2)X_t}\rmd t
\\&=:K(z_1,z_2)_t\rmd t,
\end{split}
\]
and $(t,\om,z_1,z_2)\mapsto K(\om,z_1,z_2)_t$ is a jointly progressively measurable process. 
As in Proposition \ref{prop:cov.mat.C}, we now deduce
\begin{equation}\label{eq:Cij.3/2.mod}
\begin{split}
C\p{ij}&=\sig\p2(1-\rho\p2)\int_0\p T\int_{\Sscr\p i}\int_{\Sscr\p j}\Ebb\Big[V_s\p3\partial_{v}g(s,V_s,z_1)\partial_{v}g(s,V_s,z_2)\rme\p{(z_1+z_2)X_s}\Big]\zt\p j(\rmd z_2)\zt\p i(\rmd z_1) \rmd s
\\&=
\sig\p2(1-\rho\p2)\int_0\p T\int_{\Sscr\p i}\int_{\Sscr\p j}\Ebb\Big[V_s\p3\partial_{v}g(s,V_s,z_1)\partial_{v}g(s,V_s,z_2)\Ebb\big[\rme\p{(z_1+z_2)X_s}\big|V_s\big]\Big]\zt\p j(\rmd z_2)\zt\p i(\rmd z_1) \rmd s,
\end{split}
\end{equation}
for every $i,j=1,\ldots,d$. From \eqref{eq:con.cha.X.v} with $z=z_1+z_2$ we can compute the conditional expectation on the right-hand side of \eqref{eq:Cij.3/2.mod}. The statement of the proposition follows computing the outer expectation on the right-hand side of \eqref{eq:Cij.3/2.mod} with the help of the density function of $V_t$ (cf.\ \eqref{eq:den.fun.V_t}) and the proof is complete.
\end{proof}

\begin{appendices}
\section{Moments in Heston Model}\label{subsec:mom.H}
This appendix is devoted to the proof Proposition \ref{prop:mix.mom}. We start with a preliminary lemma. The notation was introduced in \eqref{eq:xi0} and \eqref{eq:xi1}.
\begin{lemma}\label{lem:int.mix.mom}
Let $R\p j\in\Rbb$ be such that $\Ebb[\exp(2R\p jX_T)]<+\infty$, $j=1,2$. Then there exists $\ep=\ep(t)>0$ such that $\Ebb[\exp((R\p1+R\p2)X_t+(r\p1_{R\p1,R\p2}+\ep)V_t)]<+\infty$ for every $t\in[0,T]$.  In particular, 
the solution of the Riccati equation  \eqref{eq:psi.hest} starting at $(R\p1+R\p2,u)$ exists up to time $T$, for each $u\in\Rbb$ in the interval $(r\p1_{R\p1,R\p2}-\ep,r\p1_{R\p1,R\p2}+\ep)$ and $V_tH(R\p1)_tH(R\p2)_t$ is integrable, for every fixed $t\in[0,T]$.
\end{lemma}
\begin{proof}
Because of the affine structure of the Heston model, we have
\begin{equation}\label{eq:int.prod}
H(R\p1)_tH(R\p2)_t=\exp\big(r\p0_{R\p1,R\p2}\big)\exp\big((R\p1+R\p2)X_t+r\p1_{R\p1,R\p2}V_t\big).
\end{equation}
The left-hand side of the previous identity is integrable by the assumptions on $R\p1$ and $R\p2$.  Hence also the right-hand side of \eqref{eq:int.prod} is integrable. Therefore, $(R\p1+R\p2,r\p1_{R\p1,R\p2})$ belongs to the open set (see\ \cite[Theorem 10.3(b), Lemma 10.1(c)]{F09})  $M(t):=\{(u_1,u_2)\in\Cbb\p2:\ \Ebb[\rme\p{u_1X_t+u_2V_t}]<+\infty\}$, for every $t\in[0,T]$. So, for every $t\in[0,T]$, there exists $\ep=\ep(t)>0$ such that $(R\p1+R\p2,r\p1_{R\p1,R\p2}\pm\ep)\in M(t)$. Because $V_t\geq0$, using Taylor expansion of $\rme\p{\ep V_t}$, we deduce
\[\begin{split}
V_tH(R\p1)_tH(R\p2)_t&=V_t\exp\big(r\p0_{R\p1,R\p2}\big)\exp\big((R\p1+R\p2)X_t+r\p1_{R\p1,R\p2}V_t\big)\\&\leq\frac{1}{\ep}\exp\big(r\p0_{R\p1,R\p2}\big)\exp\big((R\p1+R\p2)X_t+(r\p1_{R\p1,R\p2}+\ep)V_t\big)
\end{split}
\]
and the right-hand side is integrable. The existence of the solution of the Riccati equation starting at $(R\p1+R\p2,r\p1_{R\p1,R\p2})$ follows from \cite[Theorem 2.14, Example 2.19]{KM15}. The proof is now complete.
\end{proof}
We are now ready to give a proof of Proposition \ref{prop:mix.mom}.
\begin{proof}[Proof of Proposition \ref{prop:mix.mom}]
Notice that $V_tH(z_1)_tH(z_2)_t$ is integrable. This follows from Lemma \ref{lem:int.mix.mom} and the estimate
\[
|V_tH(z_1)_tH(z_2)_t|\leq V_tH(R\p1)_tH(R\p2)_t\,,\qquad t\in[0,T].
\]
Because $H(z_j)_t=\exp\big(\phi_{T-t}(z_j,0))+\psi_{T-t}(z_j,0))V_t+z_jX_t\big)$, $j=1,2$, the Heston model being affine, we also have
\[
H(z_1)_tH(z_2)_t=\exp\big(r\p0_{z_1,z_2}\big)\exp\big((z_1+z_2)X_t+r\p1_{z_1,z_2}V_t\big),
\]
which implies the integrability of the right-hand side. From Lemma \ref{lem:int.mix.mom} this holds also for each $u\in\Cbb$ such that $u\in B_{\tilde\ep}(r\p1_{z_1,z_2})$, $\tilde\ep=\ep/2$, where $\ep=\ep(t)>0$ is as in Lemma \ref{lem:int.mix.mom} and $B_\delta(z)$ denotes the open ball in the complex plane centered in $z\in\Cbb$ and of radius $\delta>0$. From the affine formula, for $u\in B_{\tilde\ep}(r\p1_{z_1,z_2})$, we deduce
\begin{equation}\label{eq:aff.for}
\begin{split}
\Ebb\big[\exp\big(&(z_1+z_2)X_t+uV_t\big)\big]\\&=\exp\big(\phi_t((z_1+z_2),u))+(z_1+z_2)X_0+\psi_t((z_1+z_2),u)V_0\big)\,.\end{split}
\end{equation}
Now, setting $g(t,u):=\exp((z_1+z_2)X_t+uV_t)$, we have
\[
V_tH(z_1)_tH(z_2)_t=\exp\big(r\p0_{z_1,z_2}\big)\partial_{u}g(t,u)\big|_{u=r\p1_{z_1,z_2}}\,,
\]
so 
\[
\Ebb\big[V_tH(z_1)_tH(z_2)_t\big]=\exp\big(r\p0_{z_1,z_2}\big)\Ebb\Big[\partial_{u}g(t,u)\big|_{u=r\p1_{z_1,z_2}}\Big]\,.
\]
Our aim is now to exchange expectation and derivative in the previous formula. Taking the supremum over $u\in B_{\tilde\ep}\big(r\p1_{z_1,z_2}\big)$ we get
\[
\begin{split}
\textstyle\sup_{u}\displaystyle\big|\partial_{u}g(t,u)\big|&\\&=\textstyle\sup_{u}\displaystyle\big|V_t\exp\big((z_1+z_2)X_t+uV_t\big)\big|\\&=V_t\textstyle\sup_{u}\displaystyle\big(\exp\big((R\p1+R\p2)X_t+\Rea(u)V_t\big)\big)
\\&\leq V_t\textstyle\sup_{u}\displaystyle\big(\exp\big((R\p1+R\p2)X_t+\big(\Rea\big(r\p1_{z_1,z_2}\big)+\tilde\ep\big) V_t\big)\\&\leq V_t\exp\big((R\p1+R\p2)X_t+\big(r\p1_{R\p1,R\p2}+\tilde\ep\big) V_t\big)
\\&\leq \frac{2}{\ep}\exp\big((R\p1+R\p2)X_t+\big(r\p1_{R\p1,R\p2}+\ep\big) V_t\big),
\end{split}
\]
where, in the second estimation we used $\Rea(\psi_t(z))\leq\psi_t(\Rea (z))$, for every $z\in\Cbb\p2$. The last term in the previous estimation is integrable because of Lemma \ref{lem:int.mix.mom} and therefore we can exchange derivative and expectation in the following computation:
\[
\begin{split}
\Ebb\big[V_tH(z_1)_tH(z_2)_t\big]&=\exp\big(r\p0_{z_1,z_2}\big)\Ebb\Big[\partial_{u}g(t,u)\big|_{u=r\p1_{z_1,z_2}}\Big]\\&=\exp\big(r\p0_{z_1,z_2}\big)\partial_{u}\Ebb\big[g(t,u)\big]\Big|_{u=r\p1_{z_1,z_2}}.
\end{split}
\]
Using now \eqref{eq:aff.for} and computing the derivative, the statement follows because, from Lemma \ref{lem:int.mix.mom}, the solution of the Riccati equation starting at $(z_1+z_2,u)$ exists up to time $T$, for each $u$ in $B_{\tilde\ep}\big(r\p1_{z_1,z_2}\big)$.
\end{proof}
\end{appendices}

\bibliographystyle{plain}
\bibliography{bibliography}

\end{document}